\documentclass[11 pt]{article}

\usepackage[latin1]{inputenc}
\usepackage{amsmath,amsfonts,stmaryrd,amsthm,thmtools,hyperref,graphicx}

\usepackage{caption,subcaption}

\usepackage{fullpage}
\usepackage{authblk}

\newtheorem{thm}{Theorem}
\newtheorem{lemma}[thm]{Lemma}
\newtheorem{prop}[thm]{Proposition}
\newtheorem{cor}[thm]{Corollary}
\theoremstyle{definition}

\newtheorem{rem}[thm]{Remark}

\newtheorem{obs}[thm]{Observation}
\newtheorem{condition}[thm]{Condition}

\numberwithin{definition}{section}
\numberwithin{proc}{section}
\numberwithin{equation}{section}
\numberwithin{condition}{section}
\numberwithin{condition}{section}
\numberwithin{prop}{section}
\numberwithin{thm}{section}
\numberwithin{lemma}{section}
\numberwithin{rem}{section}
\numberwithin{cla}{section}
\numberwithin{obs}{section}
\numberwithin{cor}{section}

\usepackage{enumerate}

\usepackage{xcolor}
\definecolor{webgreen}{rgb}{0,.5,0}
\definecolor{Maroon}{HTML}{800000}

\hypersetup{colorlinks, breaklinks, urlcolor=Maroon, linkcolor=Maroon, citecolor=webgreen} % Set link colors

\usepackage[numbers]{natbib}
\newcommand*{\doi}[1]{doi:\,\texttt{\href{http://dx.doi.org/#1}{#1}}}

\usepackage{stmaryrd}
\usepackage{bbold} % mathbb

\newcommand{\bE}{\mathbb{E}}

\newcommand\bfd{{\mathbf d}}

\newcommand\vbfd{{\vec{\mathbf{d}}}}

\newcommand\dsD{{\mathbb D}}
\newcommand\dsE{{\mathbb E}}

\newcommand\dsG{{\mathbb G}}

\newcommand\dsN{{\mathbb N}}

\newcommand\dsR{{\mathbb R}}

\newcommand\dsZ{{\mathbb Z}}

\usepackage{physics}
\usepackage{mathtools}
\mathtoolsset{showonlyrefs}
\newcommand{\Z}{{\mathbb Z}}

\newcommand{\da}{\downarrow}
\newcommand{\ua}{\uparrow}

\newcommand\cA{{\cal A}}

\newcommand\cE{{\cal E}}
\newcommand\cF{{\cal F}}

\newcommand\cH{{\cal H}}
\newcommand\cI{{\cal I}}

\newcommand\cL{{\cal L}}

\newcommand\cN{{\cal N}}

\newcommand\cP{{\cal P}}

\newcommand\cR{{\cal R}}
\newcommand\cS{{\cal S}}
\newcommand\cT{{\cal T}}
\newcommand\cU{{\cal U}}
\newcommand\cV{{\cal V}}

\newcommand\cX{{\cal X}}

\newcommand{\E}[1]{{\mathbb E}\left[#1\right]}
\newcommand{\ee}[1]{{\mathbb E}[#1]}
\newcommand{\e}{{\mathbb E}}

\newcommand{\Vv}[1]{{\mathrm{Var}}(#1)}

\newcommand{\p}[1]{{\mathbb P}\left\{#1\right\}}
\newcommand{\pp}[1]{{\mathbb P}\{#1\}}

\newcommand\inlawHIGH{\,{\buildrel d \over \rightarrow}\,} 
\newcommand\inlaw{{\inlawHIGH}}

\newcommand{\eqd}{\,{\buildrel{\mathrm{def}} \over =}\,}

\newcommand\given[1][]{\;#1{\vert}\;}

\newcommand\cond{\;\middle|\;}

\newcommand\bigO[1]{{O\!\left( #1 \right)}}

\newcommand{\Poi}{\mathop{\mathrm{Poi}}}

\usepackage[authormarkup=none]{changes}

\definechangesauthor[name={Cai},color=red]{C}

\definechangesauthor[name={Guillem},color=blue]{G}

\DeclarePairedDelimiter{\floor}{\lfloor}{\rfloor}
\DeclarePairedDelimiter{\ceil}{\lceil}{\rceil}

\newcommand{\tin}{\mathrm{in}}
\newcommand{\tout}{\mathrm{out}}

\newcommand\hnu{{\hat \nu}}
\newcommand\hxi{{\hat \xi}}

\DeclareMathOperator{\dist}{dist}
\DeclareMathOperator{\diam}{diam}

\newcommand\din{D_{{\tin}}}
\newcommand\dinh{\hat{D}_{{\tin}}}
\newcommand\dinp{\din^{+}}
\newcommand\dinpconj{\dinh^{+}}
\newcommand\dinm{\din^{-}}
\newcommand\dnin{(D_{n})_{\tin}}
\newcommand\dninp{\dnin^{+}}
\newcommand\dninm{\dnin^{-}}

\newcommand\dout{D_{{\tout}}}

\newcommand\doutm{\dout^{-}}
\newcommand\dnout{(D_{n})_{\tout}}

\newcommand\Qnda{Q_{n}^{\downarrow}}
\newcommand\Qnua{Q_{n}^{\uparrow}}
\newcommand\hQnda{\hat{Q}_{n}^{\downarrow}}
\newcommand\hQnua{\hat{Q}_{n}^{\uparrow}}

\newcommand\Gn{{\dsG}_{n}}
\newcommand\vecGn{\vec{\dsG}_{n}}

\newcommand\GW{\mathrm{GW}}

\title{The diameter of the directed configuration model}

\author[*]{Xing Shi Cai}
\author[**]{Guillem Perarnau}
\affil[*]{\small\it Uppsala University, Sweden. Email:~{\tt xingshi.cai@math.uu.se}.}
\affil[**]{\small\it UPC. Email:~{\tt guillem.perarnau@upc.edu}.}

\begin{document}
\maketitle

\begin{abstract}
    We show that the diameter of the directed configuration model with \(n\) vertices rescaled by
    \(\log n\) converges in probability to a constant. Our assumptions are the convergence of the in-
    and out-degree of a uniform random vertex in distribution, first and second
    moment.  Our result extends previous results on the diameter of the model and applies to
    many other random directed graphs.
\end{abstract}

\section{Introduction and notations}

\subsection{The directed configuration model}

The configuration model \(\dsG_{n}\) is a uniform random multigraph on \([n]\coloneqq \{1,2,\dots,n\}\) vertices conditioned on its
degree sequence being fixed. It was introduced by Bollob{\'a}s \cite{bollobas1980} and has since
become one of the most well-studied random graph models, see, e.g., \cite{vanderhofstad2016} for
an overview.

The directed version of this model, introduced by Copper and Frieze \cite{cooper2004}, is defined analogously.  Let
\(\vbfd_n=((d^-_1,d_1^+),\dots, (d^-_n,d^+_n))\) be a bi-degree sequence with \(m_n\coloneqq
\sum_{i\in [n]} d^+_i = \sum_{i \in [n]}d^{-}_{i}\).  The directed configuration model \(\vecGn\) is
the random directed multigraph on \([n]\) obtained by first giving
\(d^-_{i}\) in half-edges (called \emph{heads}) and \(d^+_{i}\) out half-edges (called \emph{tails})
to node \(i\), and then choosing a uniform pairing of heads and tails.  In this paper, we mainly
consider the diameter of \(\vecGn\), i.e., the longest distance between two connected nodes in
\(\vecGn\).

\begin{figure}[ht]
    \centering
    \centering\includegraphics[width=0.8\textwidth]{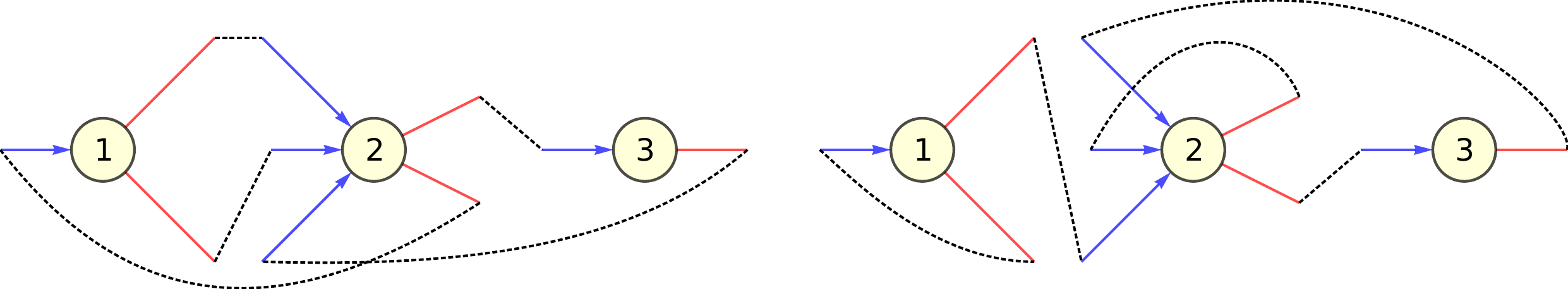}
    \caption{Examples of directed configuration model with $\vbfd_3=((1,2),(3,2),(1,1))$.}
\end{figure}

Many real-world complex networks are by nature directed. Thus, the directed configuration model has been
studied in many applied domains, such as neural networks \cite{amini2010}, 
finance  \cite{amini2013} and social
networks \cite{li2018}.

Let \(U\) be uniform random variable on \([n]\).
Let \(D_{n}=(d_{U}^{-},d_{U}^{+})\).
We denote by \(D^-_n\)  and \(D^+_n\) the marginals of \(D_n\) in each component. Let \(n_{k,\ell}\) be the number of \( (k,\ell)\) in \(\vbfd_{n}\).
 Let $\Delta_n= \max_{i\in[n]} \{d^-_i,d^+_i\}$ be the maximum  degree of $\vbfd_n$.
Consider a sequence of degree sequences \((\vbfd_n)_{n\geq 1}\). 
We assume the following:
\begin{condition}\label{cond:main}
    There exists a discrete probability distribution \(D=(D^-,D^+)\) on \(\dsZ_{\ge 0}^2\) with \(\lambda_{k,
    \ell}\coloneqq\p{D=(k,\ell)}\) such that 
    \begin{enumerate}[(i)]
        \item \(D_n\) converges to \(D\) in distribution: 
            \begin{equation}
                \lim_{n\to \infty} \frac{n_{k,\ell}}{n}  = \lambda_{k,\ell}, \qquad (k,\ell\in
                \dsZ_{\geq 0});
            \end{equation}
        \item \(D_n\) converges to \(D\) in expectation and the expectation is finite: 
            \begin{equation}
                \lim_{n\to \infty} \dsE[D^-_n] =\lim_{n\to \infty} \dsE[D^+_n] =
                \dsE[D^{-}]=\dsE[D^{+}]\eqqcolon\lambda \in (0, \infty);
            \end{equation}
        \item \(D_n\) converges to \(D\) in second moment and the second moments are finite: 
            \begin{enumerate}
                \item \(\lim_{n\to \infty} \dsE[D^-_nD^+_n] = \dsE[D^-D^+] < \infty\); 
                \item \(\lim_{n\to \infty} \dsE[(D^+_n)^2] = \dsE[(D^+)^2] < \infty\); 
                \item \(\lim_{n\to \infty} \dsE[(D^-_n)^2] = \dsE[(D^-)^2] < \infty\).
            \end{enumerate}
    \end{enumerate}
\end{condition}

\begin{rem}\label{rem:simple}
    A simple digraph (directed graph) has no self-loops and no parallel edges of the same direction
    between two vertices.  Let \(\vecGn^{s}\) be a uniform random simple digraph with degree
    sequence \(\vbfd_{n}\).  Conditioned on being simple, $\vecGn$ is distributed as
    \(\vecGn^{s}\).  Moreover, under~\autoref{cond:main}, the probability that $\vecGn$ is a
    simple digraph is bounded away from $0$,
    see~\cite{blanchet2013characterizing,janson2009probability}. Thus results that hold whp
    (with high probability) for $\vecGn$  also hold whp for \(\vecGn^{s}\).
\end{rem}

\begin{rem}
Note that removing nodes of degree \( (0,0)\) does not change the diameter of a digraph. Nonetheless, \autoref{cond:main} allows \(n_{0,0} > 0\) and \(\lambda_{0,0}>0\).
\end{rem}

An important parameter of \(D\) which governs the limit behaviour of \(\vecGn\) is
\begin{equation}\label{eq:supercri}
\nu
\coloneqq
\frac{\dsE[D^-D^+]}{\lambda}
.
\end{equation}
Note that by conditions (ii) and (iii), \(\nu<\infty\).

Cooper and Frieze~\cite{cooper2004} proved that the phase transition for the existence of a giant
strongly connected component is at $\nu=1$. Their result holds under assumptions stronger 
than~\autoref{cond:main}, including $\Delta_n\leq n^{1/12}/\log{n}$. The condition on the maximum
degree was relaxed by Graf to $\Delta_n= o(n^{1/4})$~\cite{graf2016strongly}. Throughout the paper
we assume that \(\nu>0\) and \(\nu\neq 1\).

Before stating our main result, we need to introduce two additional parameters. Let $f(z,w)$ be the bivariate
generating function of $D$. Let $s_{-}$ and $s_{+}$ be the survival probabilities of the branching
processes with offspring distributions having generating functions $\frac{1}{\lambda}\frac{\partial
f}{\partial w}(z,1)$ and  $\frac{1}{\lambda}\frac{\partial f}{\partial z}(1,w)$ respectively. Then,
we define
\begin{equation}\label{eq:hnu}
\hnu_-\coloneqq
\frac{1}{\lambda}
\frac{\partial^2 f}{\partial z\partial w} (1-s_{-},1),
\qquad
\hnu_+\coloneqq
\frac{1}{\lambda}
\frac{\partial^2 f}{\partial z\partial w}(1,1-s_{+})
,
\end{equation}
which satisfy $\hnu_-,\hnu_+\in [0,1)$.

Let the diameter of \(\vecGn\) be
\begin{equation}
    \diam(\vecGn) \coloneqq
    \max_{
        \substack{i,j \in [n]}
    }
   \{ \dist(i,j): \,  \dist(i,j)<\infty\}
    .
\end{equation}
Our main result is the following:
\begin{thm}\label{thm:diam}
    Suppose that $(\vbfd_n)_{n \ge 1}$ satisfies~\autoref{cond:main}. 
    \begin{enumerate}[\normalfont(i)]
        \item The supercritical case: 
            If \(\nu > 1\), then 
            \begin{equation}\label{eq:diam}
                \frac{
                    \diam(\vecGn)
                }{
                \log n}
                \to
                \frac{1}{\log (1/\hnu_{+})}
                +
                \frac{1}{\log (1/\hnu_{-})}
                +        
                \frac{1}{\log \nu}
                ,
            \end{equation}
            in probability, where we use the convention that \(1/\log(1/0)=0\).
        \item The subcritical case: If \(0<\nu<1\), then
            \begin{equation}\label{eq:diam:sub}
                \frac{
                    \diam(\vecGn)
                }{
                \log n}
                \to
                \frac{1}{\log (1/\nu)}
                ,
            \end{equation}
            in probability.
    \end{enumerate}
\end{thm}

In the supercritical case \eqref{eq:diam}, there are three terms that contribute to the diameter. The first term is given by vertices whose out-neighbourhoods neither expand nor die for many steps (\emph{thin
out-neighbourhoods}), and the second one is the analogue for in-neighbourhoods (\emph{thin
in-neighbourhoods}). Due to the symmetry in~\eqref{eq:supercri}, the typical expansion rate of the
in- and out-neighbourhoods of a vertex is the same. However, conditioned on the rare event of
``having a thin neighbourhood'', the expansion rate is different (see~\autoref{sec:app} for some
particular examples).
The last term in~\eqref{eq:diam} corresponds to the typical distance between a
thin in- and out-neighbourhood. The case $\p{D^+=0}=0$ is of particular interest (and similarly for
$\p{D^-=0}=0$). If additionally $\p{D^+=1}=0$, then almost all vertices in $\vbfd_n$ have out-degree at least $2$
and there are no thin out-neighbourhoods, so $\hnu_+=0$ and the first term in~\eqref{eq:diam} disappears.
Otherwise $\p{D^+=1}>0$, one can check that $\hnu_{+}=\frac{1}{\lambda}\p{D^+=1}$ and the thin out-neighbourhoods are directed
paths.

In the subcritical case, we have \(\hnu_{+}=\hnu_{-}=\nu\).  In other words, thin in- and
out-neighbourhoods of length  \(\log_{1/\hnu_{\pm}} n\) still exist whp, but instead of
expanding to large size they die before intersecting each others.  Thus there is only one term in
\eqref{eq:diam:sub}, which comes from both long in- and out-neighbourhoods.

The proof of~\autoref{thm:diam} is based on the analysis of a BFS (Breadth First Search)
edge-exploration process of the
out-neighbourhoods of a given tail in $\vecGn$ (and similarly for the in-neighbourhoods of heads)
and its coupling with the corresponding branching process. Convergence in \autoref{cond:main} is
usually required in this setting and is necessary to ensure that we can couple the exploration
process with a Galton-Watson tree with offspring obtained from $D$. It would be interesting to see
if one could drop the condition \(\nu < \infty\), as in the case of the undirected configuration
model~\cite{fernholz2007diameter} (see~\autoref{sec:related}).

We make no assumption on the rate of convergence in~\autoref{cond:main}, thus, we cannot determine
the second order term of  \(\diam(\vecGn)\). Under explicit convergence rate assumptions, it might be possible to find the second order term, as
in~\cite{caputo2019a,riordan2010}.

\subsection{Previous results on distances in configuration models}\label{sec:related}

We first discuss the previous results obtained for the undirected configuration model $\Gn$ with
degree sequence $\bfd_n = (d_{1},\dots,d_{n})$.
Bollob\'as and Fernandez de la Vega~\cite{bollobas1982diameter} determined the asymptotic diameter
of random regular graphs; that is, the case where $\bfd_n$ contains only a constant. Fernholz and
Ramachandran \cite{fernholz2007diameter} obtained an asymptotic expression
for the diameter of $\Gn$.

To state the result in \cite{fernholz2007diameter}, some notation is needed: Let \({D}_{n}\)
    be chosen uniformly at random from \(\bfd_{n}\).  Let \(D\) be a discrete random variable on
    \(\dsZ_{\ge 0}\) with
distribution $\lambda_k\coloneqq \p{D=k}$. Let \(n_{k}\) be the number of \(k\) in \(\bfd_{n}\).  
Let \(D^{*}\) be a random variable on \(\dsZ_{\ge 0}\) with distribution 
\(\lambda_{k}^{*} \coloneqq \p{D^{*}=k} = (k+1)\lambda_{k+1}/\E{D}\); \(D^{*}\) is the \emph{size-biased} distribution of \(D\). 
Let \(\hat{D}\) be the \emph{conjugate} of \(D^*\) (see \autoref{sec:kesten}) and let
\(\hat{\nu}=\ee{\hat{D}}\).

\begin{rem}
    The size-biased distribution of \(D\) is sometimes defined as
        \(D^{s}\) with \(\pp{D^{s} = k} = k \lambda_k /\ee{D}\). Note that $D^s=D^*+1$. We use $D^*$ for the sake of convenience.
\end{rem}

\begin{thm}\label{thm:diam:und}
    Assume that $D_n\to D$ in distribution, first and second
    moment, $\E{D}<\infty$ and $\lambda_1>0$, and that $\nu\coloneqq \E{{D}^2}/\E{D}{>1}$. 
    Then 
    \begin{equation}
        \frac{
            \diam(\Gn)
        }{
        \log n}
        \to
        \frac{2}{\log (1/\hnu)}
		+        
        \frac{1}{\log \nu}
        ,
    \end{equation}
    in probability.
\end{thm}
The case $\lambda_1=0$ is discussed in~\cite[Theorem 7.16]{vanderhofstad2020}. 
Under
the extra conditions $n_1=0$ when $\lambda_1=0$ and $n_2=0$ when $\lambda_2=0$, \autoref{thm:diam:und}
extends to
\begin{equation}\label{eq:diam:und}
    \frac{
        \diam(\Gn)
    }{
    \log n}
    \to
    \frac{2\cdot\mathbb{1}[\,\lambda_1> 0]}{\log (1/\hnu)}
    +   
        \frac{\mathbb{1}[\,\lambda_1=0,\, \lambda_2> 0]}{\log (1/\lambda_1^*)}
        +     
    \frac{1}{\log \nu}
    ,
\end{equation}
in probability.
\autoref{thm:diam}~(i) can be seen as the directed analogue of~\eqref{eq:diam:und}. The main difference is that, in the directed case, the first term in~\eqref{eq:diam:und} splits into two (corresponding to thin in- and out-neighbourhoods which behave differently), and that there is no exceptional behaviour in the case $\lambda_1=0, \lambda_2>0$. The proof of~\autoref{thm:diam} draws
similarities with the proof of~\autoref{thm:diam:und}, based on the analysis of a BFS exploration process.
However, the proof of~\autoref{thm:diam:und} restricts the exploration to the $2$-core of $\Gn$ and thus, it heavily relies on previous understanding of the size and degree distribution of the $2$-core, which are not known in the directed setting.
    
Refinements of~\autoref{thm:diam:und} have been obtained for particular degree sequences. For
example, Riordan and Wormald~\cite{riordan2010} proved that for every $c>1$ there exists $\eta_c>0$
such that the binomial random graph $\dsG(n,p)$ satisfies $\diam(\dsG(n,c/n))= \eta_c \log n+O(1)$
whp. (In this case $D$ is distributed as a Poisson with mean $c$.) We will use some
of the ideas introduced in~\cite{riordan2010}, to analyse the exploration process on \(\vecGn\)
without taking into consideration its core.
    
For degree sequences with $\ee{D^2}=\infty$ and $\lambda_1=\lambda_2=0$,~\autoref{thm:diam:und} implies the weak upper bound on the
diameter $o(\log n)$. In the particular case of power-law distributions with exponent $\tau\in (2,3)$ and
provided that the minimum degree is at least $3$, van der Hofstad  showed that the diameter is of
order $\log\log n$~\cite[Theorem 7.17]{vanderhofstad2020}.

Recently, there has been some progress on the diameter of the supercritical directed configuration
model. Caputo and Quattropani~\cite{caputo2019a} determined the asymptotic behaviour of the diameter
of $\vecGn$ provided that $2\leq d_i^-,d_i^+=O(1)$. One of the motivations to study the diameter of
directed random graphs is its close connection to the properties of a random walk on it. For
instance, in~\cite{caputo2019a} the authors used their results on neighbourhood expansion to
determine the extremal values for the stationary distribution of a random walk in $\vecGn$, with
implications on its cover time. Typical values of the stationary distribution were previously
obtained by Bordenave, Caputo and Salez~\cite{bordenave2016,bordenave2018} as an intermediate step
to bound the mixing time of the random walk. Finally, typical distances in $\vecGn$ have been
recently studied by van der Hoorn and Olvera-Cravioto~\cite{hoorn2018}. We are not aware of any
result in the subcritical regime.

A related model is the $d$-out random digraph. In this model, each node is given a set of $d$
out-edges that connect independently to other vertices. This model is of particular interest since
it provides a way to study random Deterministic Finite Automata. Penrose studied the emergence of a
linear order strongly connected component~\cite{penrose2016strong} and its diameter was determined
by Addario-Berry, Balle and the second author~\cite{addario2015diameter}. In \cite{cai2017b}, the
first author and Devroye studied the diameter outside the giant strongly connected component and other properties of the \(d\)-out model.

\autoref{thm:diam} implies the results on the asymptotic behaviour of the diameter previously obtained in~\cite{addario2015diameter,caputo2019a}.

\subsection{Organisation of the paper}

The paper is organised as follows. In~\autoref{sec:small:nodes} we obtain results on the number of
edges incident to small sets of nodes. In~\autoref{sec:branching} we study rare events in branching
processes. ~\autoref{sec:size:bias} introduces the in- and out-size-biased distributions. We present
an edge-BFS-exploration process in~\autoref{sec:coupl} and couple it with the corresponding branching
process. In~\autoref{sec:tower} we find thin in- and out-neighbourhoods with large depth that will
give rise to the first two terms in~\eqref{eq:diam} and the term in~\eqref{eq:diam:sub}. Typical distances between large sets of
edges are studied in~\autoref{sec:typ:dist}, giving rise to the last term in~\eqref{eq:diam}. \autoref{thm:diam} is proved in~\autoref{sec:final}. Finally, in \autoref{sec:app}, we
present some applications.

\section{Small sets of nodes}\label{sec:small:nodes}

At various stages of our proof, we will need the fact that~\autoref{cond:main} implies that any small set of nodes is incident to a small number
of half-edges. We state this formally in \autoref{lem:s}.

{Let \(X\) and \(Y\) be random variables. We say that \(X\) is \emph{stochastically
dominated} by \(Y\) if \(\p{X \geq z } \le \p{Y\geq z}\) for all \(z \in \dsR\), and we denote it by \(X \le_{\text{st}} Y\).}
First we need the following simple statement whose simple proof we omit.
\begin{lemma}\label{lem:dct}
Let \(X_n\geq 0\) and \(Y_n\geq 0\) be two sequences of
random variables such that \(X_{n} \inlaw X\) and \(Y_{n} \inlaw Y\). Assume that
\(X_n\leq _{\text{st}}Y_n\)
and \(\mathbb{E}[Y_n]\to \mathbb{E}[Y]\).
Then \( \E{X_{n}} \to \E{X} \).
\end{lemma}

For \(\cS \subseteq [n]\), let
\begin{equation}
    d_{\cS}(i,j) \coloneqq  \sum_{v \in \cS} (d_v^-)^{i} (d_v^+)^{j}
    .
\end{equation}
\begin{lemma}\label{lem:s}
{Assume \autoref{cond:main}}. Let \(s_n=o(n)\) be a sequence of numbers. Then uniformly for all
\(\cS\) with \(\abs{\cS} \le s_{n}\),
\begin{equation}\label{eq:nuI}
    d_{\cS}(1,1)
    =o(n)
    ,
\end{equation}
and
\begin{equation}\label{eq:nuI:1}
    d_{\cS}(1,0)
    =o(\sqrt{s_{n}n })
    ,
    \quad
    d_{\cS}(0,1)
    =o(\sqrt{s_{n}n })
    .
\end{equation}
\end{lemma}

\begin{proof}
    We first show \eqref{eq:nuI} by a coupling argument.  Recall that
    \(D_{n}=(d_{U}^{-},d_{U}^{+})\) where \(U\) is a uniform random variable on \([n]\).  Define a
    random variable \(D_{n,*}\) on \(\Z_{\geq 0}^2\) by \(D_{n,*} = D_{n}\) if \(U \notin \cS\) and
    \(D_{n,*}={(0,0)}\) otherwise.  Then 
    \begin{equation}
        \p{D_{n,*} \ne D_{n}} \le \frac{{s_{n}}}{n} = o(1).
    \end{equation}
    By \autoref{cond:main}, the above implies that \(D_{n,*} \inlaw D\) and \(D_{n,*}^{-}D_{n,*}^{+}
    \inlaw D^{-}D^{+}\).  Note that our coupling ensures \(D_{n,*}^{-} D_{n,*}^{+}\le_{\text{st}}
    D_{n}^{-} D_{n}^{+}\) which converges in mean and in distribution to \(D^{-}D^{+}\).  Thus it
    follows from \autoref{lem:dct} that \({\ee{D_{n,*}^{-}D_{n,*}^{+}}} \to \E{D^{-}D^{+}}\) by
    taking \(X_n=D_{n,*}^{-} D_{n,*}^{+}\) and \(Y_{n}=D_{n}^{-}D_{n}^{+}\).  Therefore,
    \begin{equation}
        d_{\cS}(1,1)
        =
         \sum_{v\in \cS} d_v^+ d_v^-
        =
         \sum_{v\in [n]} d_v^+ d_v^- - \sum_{v\notin \cS} d_v^+ d_v^-
        =
        n\ee{D_{n}^{-}D_{n}^{+}-D_{n,*}^{-}D_{n,*}^{+}}
        =
        o(n)
        .
    \end{equation}

    For \eqref{eq:nuI:1}, first note that \autoref{lem:dct} together with \autoref{cond:main}
    also implies that \(\e[(D_{n,*}^{-})^{2}] \to \e[(D^{-})^{2}]\).
    Thus
    \begin{equation}
        \sum_{v \in \cS}
        (d_{v}^{-})^{2}
        =
        \sum_{v \in [n]}
        (d_{v}^{-})^{2}
        -
        \sum_{v \notin \cS}
        (d_{v}^{-})^{2}
        =
        n \ee{({D_n^-})^{2}-(D_{n,*}^{-})^{2}}
        =
        o(n)
        .
    \end{equation}
    It follows from Cauchy-Schwarz inequality that
    \begin{equation}
        d_{\cS}(1,0)
        =
        \sum_{v \in \cS}
        d_{v}^{-}
        \le
        \sqrt{
            \Bigl(
                \sum_{v \in \cS}
                1
            \Bigr)
            \Bigl(
                \sum_{v \in \cS}
                (d_{v}^{-})^{2}
            \Bigr)
        }
        =
        o \left(\sqrt{s_{n} n}  \right)
        .
    \end{equation}
    The same argument works for \(d_{\cS}(0,1)\).
\end{proof}

For \(\cI \subseteq [n]\), define
    \begin{equation}\label{eq:nu:I}
        \nu_{\cI} \coloneqq 
        \frac{d_{\cI}(1,1)}{m_{n}}
		=
        \frac{\sum_{v \in \cI} d_v^- d_v^+}{m_{n}}
        .
	\end{equation}

\begin{cor}
    \label{cor:s}
    Under the hypothesis of \autoref{lem:s} and uniformly for all \(\cI\subseteq [n]\) with \(\abs{\cI} \geq  n -s_{n}\), we have
    \(\nu_{\cI} = (1+o(1))\nu\).
\end{cor}

\begin{proof}
   The corollary follows from \eqref{eq:nuI} with \(\cS = [n] \setminus \cI\).
\end{proof}

\begin{cor}
    \label{cor:D}
  Under \autoref{cond:main}, we have $\Delta_n = o(\sqrt{n}).$
\end{cor}

\begin{proof}
    Let \(\cS\) be the set containing only a node with maximum out/in degree and apply \eqref{eq:nuI:1}.
\end{proof}

\section{Branching processes}\label{sec:branching}

Let \(\xi\) be a random variable on \(\dsZ_{\ge 0}\) and let
\((\xi_{i,t})_{i\ge 1, t \ge 0}\) be iid (independent and identically distributed) copies of \(\xi\).
The branching process, also known as the Galton-Watson tree, \((X_t)_{t \ge 0}\) with offspring
distribution \(\xi\)
is defined by
\begin{equation}\label{eq:BP}
    X_{t}= 
    \begin{cases}
        1
        &
        \qquad (t =0)
        \\
        \sum_{i=1}^{X_{t-1}} \xi_{i,t-1}
        &
        \qquad (t \ge 1)
    \end{cases}
\end{equation}
Let $h$ be the probability generating function of $\xi$, i.e., \(h(z)=\sum_{i \ge 0} \p{\xi = i}
z^{i}\). Then
$$
\nu_{\xi}\coloneqq h'(1)=\E{\xi}. 
$$
\subsection{Supercritical branching process}

In this subsection we will assume that $\nu_{\xi} \in (1,\infty)$, usually referred to as $ (X_t)_{t
\ge 0}$ being \emph{supercritical}.

\subsubsection{Convergence of branching processes}
The sequence \(\nu_{\xi}^{-t}X_t\) is a martingale which, provided that \(\nu_{\xi}<\infty\), converges almost
surely to a random variable \(W\). However, stronger conditions are needed to show that \(W\) is
non-degenerated. The following is due to Kesten and Stigum (see \cite[pp.\ 24--29]{athreya1972} for a proof):
\begin{thm}\label{thm:kesten}
    Let \( (X_t)_{t \ge 0}\) be the branching process defined above. 
    Assume that \(\p{\xi = i} \ne 1\) for all \(i \in \dsN\).
    Let
    \(W= \lim_{t\to \infty} \nu_{\xi}^{-t} X_t\) be the rescaled limiting random variable. 
    \begin{enumerate}[\normalfont(i)]
        \item If \(\dsE[\xi \log_+ (\xi)]<\infty\), then \(\dsE[W]=1\) and \(W\) is absolutely
            continuous on \( (0, \infty)\).
        \item If \(\dsE[\xi \log_+ (\xi)]=\infty\), then \(W=0\) almost surely.
    \end{enumerate}
\end{thm}

In the case that \(\dsE[\xi \log_+(\xi)]<\infty\), since almost sure convergence implies
convergence in probability, it follows that for every fixed \(0<c_1  <c_2\),
\begin{align}\label{eq:cond_non-deg}
    \inf_t \p{c_1 \nu_{\xi}^t \leq X_t\leq c_2 \nu_{\xi}^t} > 0,
\end{align}
where the infimum is taken over all \(t\) such that \([c_1\nu_{\xi}^t,c_2\nu_{\xi}^t] \cap \dsN\) is non-empty.

{If \(\dsE[\xi \log_{+}(\xi)]=\infty\), the growth of \(X_t\) is not necessarily of order \(\nu_{\xi}^t\)
any more.  However, there always exists a normalization sequence \(\{m_{\xi,t}\}_{t\geq 0}\) with
\(\lim_{t\to \infty} (m_{\xi,t})^{1/t} = \nu_{\xi}\) and \(m_{\xi,t}=O(\nu_{\xi}^t)\), such that
\(X_t/m_{\xi,t}\) converges almost surely to a non-degenerate limit}~(see \cite[pp.\
30]{athreya1972} for a proof and further reference).
In words, the exponential growth rate of \(X_{t}\) is still \(\nu_{\xi}\), but there might be subexponential fluctuations that slow it down.  
Again, almost sure convergence implies that for every \(0<c_1  <c_2\),
\begin{align}\label{eq:cond_deg}
\inf_{t} \p{c_1 m_{\xi,t} \leq X_t\leq c_2 m_{\xi,t}} > 0,
\end{align}
where the infimum is taken over all \(t\) such that \([c_1m_{\xi,t},c_2m_{\xi,t}] \cap \dsN\) is non-empty.

If \(\p{\xi=\nu_{\xi}}=1\) for some \(\nu_{\xi} \ge 2\), then \autoref{thm:kesten} does no longer apply.
Nonetheless, we can define \(m_{\xi,t} = \nu_{\xi}^t\) and \eqref{eq:cond_deg} still holds for any
\(0 < c_{1} \le 1  \leq  c_{2}\).

\subsubsection{Duality and conditioned branching process}\label{sec:kesten}

The \emph{survival
    probability} of \((X_t)_{t \ge 0}\) is defined by
\begin{equation}
    s \coloneqq\p{X_t>0,\, \text{for all }t\geq 1}.
\end{equation}
It is well-known that \(s > 0\) if and only if \(\nu_{\xi} > 1\), see, e.g., \cite[Theorem 3.1]{vanderhofstad2016}.

For $s<1$, the \emph{conjugate probability
distribution} of \(\xi\), \(\hxi\) is defined by
\begin{equation}\label{eq:conj}
    \pp{\hxi=\ell} \coloneqq (1-s)^{\ell-1}\p{\xi=\ell}.
\end{equation}
The definition of \(\hxi\) can be extended to encompass the case \(s=1\) by taking the limit of
\eqref{eq:conj} as \(s \uparrow 1\). In other words, when \(s = 1\), 
\begin{equation}\label{XHNOP}
    \pp{\hxi = 1}  = \pp{\xi =1}, \qquad \pp{\hxi = 0} = 1-\p{\xi = 1}.
\end{equation}
Note that
\begin{equation}\label{PXNEC}
    \hnu_{\xi} \coloneqq \e[\hxi] =  h'(1-s) \in [0, 1).
\end{equation}
By the assumption that \(\nu_{\xi}>1\), we have \(s>0\). Thus \(\hnu_{\xi} < 1\)  and \(\hnu_{\xi} = 0\) if
and only if \(\p{\xi\leq 1} = 0\).

\begin{lemma}\label{lem:conj}
    Let \(\xi_{n}\) be a sequence of random variables such that \(\xi_{n} \to \xi\) in distribution
    and in expectation. 
    Then \(\hat{\xi}_{n} \to \hat{\xi}\) in distribution and in expectation.
\end{lemma}

\begin{proof}
    Let \(s_{n}\) and \(s\) be the survival probabilities of the branching processes with offspring
    distributions
    \(\xi_{n}\) and \(\xi\) respectively. It follows from Lemma~4.5 in \cite{bordenave2016} that
    \(s_{n} \to s\).

    Thus by the definition of conjugate distribution \eqref{eq:conj},  we have \(\hat{\xi}_{n} \to \hat{\xi}\) in distribution.
    To see that the convergence is also in expectation, note that
    \(\hat{\xi}_{n} \le_{\text{st}} \xi_{n}\). 
    Thus we can apply \autoref{lem:dct} with \(X_{n} = \hat{\xi}_{n}\) and \(Y_{n} = \xi_{n}\).
\end{proof}
    
An important property of supercritical branching processes is \emph{duality} \cite[Theorem
3.7]{vanderhofstad2020}:
\begin{thm} \label{thm:branching}
    Let \( (X_{t})_{t \ge 0}\) be as in \autoref{thm:kesten} and let \(s\) be its survival probability.
    If \(s < 1\), then the branching process \( (X_{t})_{t \ge 0}\) conditioned on extinction is distributed as a branching process
    with offspring distribution \(\hxi\).
\end{thm}

{The next key result allows us to estimate the probability of certain rare events in branching
processes. It generalises a result of Riordan and Wormald~\cite[Lemma 2.1]{riordan2010}, who proved
it for Poisson distributed offsprings.}

\begin{thm}
    \label{thm:BP}
    Let \( (X_t)_{t \ge 0}\) be a branching process with offspring distribution \(\xi\) with
    \(\nu_{\xi}\in(1,\infty)\).  
    Let  \({t_\xi(\omega)}\coloneqq \inf \{t\geq 0: m_{\xi,t'}\geq \omega \text{ for all }t'\geq t\}\). 
    \begin{enumerate}[\normalfont(i)]
        \item If $\p{\xi \le 1}>0$, then there exist constants \(c, C>0\) depending on $\xi$ such that
            \begin{equation}\label{DMEXN}
                c\hnu_{\xi}^{t} \leq \p{\cap_{r=1}^{{t}} [0<X_r<\omega]}\leq C \hnu_{\xi}^{t-t_\xi(\omega)},
                \qquad
                (t \ge 1, \omega \ge t).
            \end{equation}

        \item If $\p{\xi \le 1}=0$, then \(\hnu_{\xi}=0\) and
            there exists $c>0$ depending on $\xi$ such that
            \begin{equation}\label{eq:tower:prob}
                \p{\cap_{r=1}^{{t}} [0<X_r<\omega]}\leq \exp\{-c 2^{t-t_\xi(\omega)}\},
                \qquad
                (t \ge 1, \omega \ge 1).
            \end{equation}
    \end{enumerate}
\end{thm}

\begin{rem}\label{rem:seneta}
    If \(\E{\xi \log_{+}(\xi)}<\infty\), by \autoref{thm:kesten}, we may choose
    \(m_{\xi,t}=\nu_{\xi}^t\) and \(t_\xi(\omega)=  \lceil\log_{\nu_{\xi}} \omega\rceil\).  In fact,
    the result of Seneta \cite{seneta1968a} implies that \(t_\xi(\omega)= (1+o(1))\log_{\nu_{\xi}}
    \omega\) if \(\E{\xi \log_{+}(\xi)}=\infty\).
\end{rem}

\begin{proof}[Proof of the upper bound in (i)]
It suffices to provide an upper bound for the probability of the event \([0<X_t<\omega]\). Since the proof
follows the same ideas as the proof of Lemma 2.1 in \cite{riordan2010}, we omit the details that
are identical to the aforementioned lemma. 

Define \(r_t\coloneqq\p{X_t<\omega\mid X_t>0}\) and write \(t(\omega)=t_\xi(\omega)\).
Equation~\eqref{eq:cond_deg} and the definition of \(t(\omega)\) imply that 
\begin{equation}\label{AGJJP}
    \begin{aligned}
    \pp{{X_{t(\omega)} \ge \omega} \given X_{t(\omega)}>0} 
    &
    \ge 
    \p{X_{t(\omega)} \ge\omega}
    \\
    &
    \geq  \p{X_{t(\omega)}  \ge m_{\xi,t(\omega)}} 
    \\
    &
    \ge
    \p{m_{\xi,t(\omega)} \le X_{t(\omega)} \le 2 m_{\xi,t(\omega)}}>c_0,
    \end{aligned}
\end{equation}
for some \(c_0>0\), which implies that \(r_{t(\omega)}<1-c_0\). 

Let \(s\) be the survival probability of $X_t$ and let \(S(t)\subseteq X_1\) be the set of children
of the initial particle that have progeny in \(X_t\).
Thus, as in~\cite{riordan2010}, to show that \(r_t\) has exponential decrease with basis \(\hnu_{\xi}\) for \(t\geq
t(\omega)\), it suffices to show that
\begin{equation}\label{JKZUU}
\p{|S(t)|=1| \given[\big] |S(t)|\geq 1} = \hnu_{\xi} (1+O(q(s) \hnu_{\xi}^t))
,
\end{equation}
for some function \(q\). This directly implies the upper bound in (i) since \(\p{0<X_{t}<\omega} \le r_{t}\).
We prove \eqref{JKZUU} in the following.

We first consider the case $s<1$. Recall that \(s>0\) and let
\(s_t=\p{X_t>0}\).  We can assume that $t$ is large enough
    with respect to $s$, as we can set $C$ large enough with respect to $\hnu_{\xi}$ so the bound holds
trivially for small values of $t$. Using Markov inequality and~\autoref{thm:branching},
\begin{equation} \label{eq:s:UB}
    \begin{aligned}
        s \leq s_{t} &= \p{(X_r)_{r \ge 0}\text{ survives}}+\p{X_{t}>0,(X_r)_{r \ge 0}\text{  extinguishes}}\\ 
        & = s+(1-s)\p{X_{t}>0\mid (X_r)_{r \ge 0} \text{ extinguishes}} \\
        & \leq s+(1-s)\hnu^t_{\xi}.
    \end{aligned}
\end{equation}

Conditioning on \(X_{1}\), the events \([x \in S(t)]\) for \(x \in X_{1}\) happen independently with
probability \(s_{t-1}\). Thus, the random
variable \(|S(t)|\) has the distribution of a \(s_{t-1}\)-thinned version of \(\xi\) and 
\begin{align}\label{eq:st1}
\p{|S(t)|=1| \given[\big] |S(t)|\geq 1} 
=
\frac{\p{|S(t)|=1| } }{\p{|S(t)| \geq 1|}}
= \frac{s_{t-1}h'(1-s_{t-1})}{1-h(1-s_{t-1})} .
\end{align}

We use Taylor expansion to approximate \(h(1-s_{t-1})\) and \(h'(1-s_{t-1})\) around \(1-s\). First note that for every \(m\geq 1\), the \(m\)-th derivative of \(h\) is bounded at \(1-s\)
\begin{equation}
h^{(m)}(1-s) 
= \sum_{\ell\geq 0} (\ell)_m (1-s)^{\ell-m} \p{\xi=\ell}\leq \sum_{\ell\geq 0} (\ell)_m (1-s)^{\ell-m} = m! s^{-m-1} ,
\end{equation}
where \( (\ell)_{m} \coloneqq  \ell(\ell-1)\dots(\ell-m+1).\) 
Using Taylor's theorem, we have, uniformly for all \(t \ge 1\)
\begin{align*}
    h(1-s_{t-1})&= h(1-s)+ h'(1-s)(s-s_{t-1}) + O((s-s_{t-1})^2)= 1-s + O(\hnu_{\xi}^t)
,
\\
h'(1-s_{t-1})&= h'(1-s)+ h''(1-s)(s-s_{t-1}) + O((s-s_{t-1})^2)=\hnu_{\xi} +
O(\hnu_{\xi}^t) ,
\end{align*}
where we use \(h(1-s)=1-s\), \(h'(1-s)=\hnu_{\xi}\) and~\eqref{eq:s:UB}.
Using these estimates and~\eqref{eq:st1}, we obtain \eqref{JKZUU}.

In the case $s=1$, the event $[\abs{S(t)}\geq 1]$ holds almost surely and  \eqref{JKZUU} still holds
since
\[
    \p{|S(t)|=1| \given[\big] |S(t)|\geq 1} 
    =
    \p{|S(t)|=1}
    =\p{\xi=1}=\hnu_{\xi}.
    \qedhere
\]
\end{proof}

\begin{proof}[Proof of (ii)]
    Let \(r_{t}\) and \(c_{0}\) be as in the proof of the upper bound in (i).  
    In this case, the event \([\abs{S(t)} \ge 2]\) holds almost surely.
    Thus
    \begin{equation}\label{THAVK}
        r_t
        \leq 
        \p{|S(t)|\geq 2}r_{t-1}^2 = r_{t-1}^2,
    \end{equation}
    and $r_t \leq (r_{t(\omega)})^{2^{t-t(\omega)}}< (1-c_0)^{2^{t-t(\omega)}}$. 
\end{proof}
\medskip

\begin{proof}[Proof of the lower bound in (i)]
    Let \((X_r^*)_{r \ge 0}\subseteq (X_r)_{r \ge 0}\) be the subprocess of the elements that have some surviving progeny.
For the lower bound, consider the following events: 
\begin{equation}
E_1=[X_{t}^*=1]
,
\qquad
E_2 = \left[\cap_{r=1}^t [0<X_r<\omega]\right]
.
\end{equation}
{Instead of lower bounding the probability of $E_2$, we will show a lower bound for the  probability of \(E \coloneqq E_{1} \cap E_{2}\). Write}
\begin{equation}\label{eq:E}
\p{E} = \p{E_1} \p{E_2 \mid E_1}. 
\end{equation}

We start by computing \( \p{E_1} \).  Conditioning on that an element of \(X_t\) belongs to
\(X_t^*\) is equivalent to conditioning on that the progeny of at least one of its children
survives. So, conditional on \(X_t\) surviving, \(X_t^*\) is a
branching process with offspring distribution \(\xi^*\), the \emph{$s$-thinned} version of \(\xi\)
conditioning on being at least \(1\). In other words, for $\ell\geq 1$
\begin{equation}
    \p{\xi^*=\ell} 
    = \frac{\sum_{m \ge \ell} \p{\xi=m} \binom{m}{\ell} s^\ell (1-s)^{m-\ell} }{s}
    = \frac{s^{\ell-1} h^{(\ell)}(1-s)}{\ell!}
    \; ,
    \qquad 
\end{equation}
and in particular, by definition,%~\autoref{thm:branching}},
\begin{align}\label{eq:Z}
    \p{\xi^*=1} = h'(1-s) = \hnu_{\xi}.
\end{align}
We conclude that
\begin{equation}\label{eq:E1}
\p{E_1} 
=
\p{[(X_r)_{r \ge 0}\text{ survives}]\cap \left[\cap_{r=1}^ t [X^*_r=1]\right]}
=
s \hnu_{\xi}^{t}
.
\end{equation}

If $\p{\xi=0}=0$, then $X_t^*=1$ implies that $X_r=1$ for all $r\leq t$. Therefore, $\p{E_2\cond
E_1}=1$ and we are done. Thus, we assume that $\p{\xi=0}>0$, and so $s<1$.

Conditioning on \(E_{1}\), the tree can be seen as the main branch of length $t$ (the part that has surviving
progeny), with \(t\) independent branching processes with offspring distribution \(\hxi\)
attached to each node on the path. Thus uniformly for \(0 \le r \le t\), 
\begin{equation}\label{eq:X:r:mean}
    \E{X_{r} \given E_{1}}
    =
    1 + \sum_{j=1}^{r} \hnu_{\xi}^{j}
    =
    O(1)
    ,
\end{equation}
and
\begin{equation}\label{eq:X:r:var}
    \Vv{X_{r} \given E_{1}}
    =
    \sum_{j=1}^{r} 
    \frac{\Vv{\hxi} \hnu^{j-1}(\hnu^{j}-1)}{\hnu-1}
    =
    O( 
    \Vv{\hxi}
    )
    ,
\end{equation}
where we use the moment formula in \cite[pp. 4]{athreya1972}.

By Theorem I.12.3 in \cite{athreya1972}, the probability generating function of \(\hat \xi\) is
\begin{equation}\label{eq:sD:pgf}
    \hat h(z)
    \coloneqq
    \E{ z^{\hxi}}
    = 
    \frac{h(z (1-s))}{1-s}
    .
\end{equation}
Since the radius of convergence of \(h(z)\) is at least \(1\) and \((1-s) \in (0,1)\), we have
\(\hat h''(1) < \infty\). In other words, \(\Vv{\hxi}<\infty\).

Therefore, it follows from Chebyshev's inequality that
\begin{equation}\label{eq:X:r:up}
    \p{E_{2}^c \given E_{1}}
    \le
    \sum_{j=0}^{t}
    \p{X_{j} \ge \omega \given E_{1}}
    =
    \bigO{
        \sum_{r=1}^{t} 
        \frac{\Vv{X_{r}|E_{1}}}{\omega^{2}}
    }
    =
    \bigO{
        \frac{t \Vv{\hat \xi}}{\omega^{2}}
    }
    =
    O(t^{-1})
    .
\end{equation}
The lemma follows immediately by putting  \eqref{eq:E1} and \eqref{eq:X:r:up} into \eqref{eq:E}.
\end{proof}

\subsection{Subcritical branching process}

If $\nu_\xi\in (0,1)$, then $s=0$ and $\hnu_\xi=\nu_\xi$.
Thus, the following theorem shows that
the depth of thin supercritical branching processes is close to the depth of subcritical
processes. This has already been observed in~\cite{fernholz2007diameter}.

\begin{thm}\label{thm:BP:sub}
    Let \( (X_t)_{t \ge 0}\) be a branching process with offspring distribution \(\xi\) with
    \(\nu_{\xi}\in(0,1)\).  Then
    \begin{equation}\label{IFWMJ}
        \lim_{t\to \infty} \p{X_{t} > 0}^{1/t} = \nu_{\xi}.
    \end{equation}
    Moreover, letting \(Y_{t} = \sum_{i=0}^{t} X_{i}\),
    for all \(\omega(t)\) such that \(\omega(t)/t = \infty\) as \(t \to \infty\), we
    have
    \begin{equation}\label{PAPPP}
        \lim_{t\to \infty} 
        \p{\left[ Y_{t} \le \omega(t)\right] 
            \cap 
            \left[ 
            X_{t} > 0
            \right]
        }^{1/t}
        = 
        \nu_{\xi}
        .
    \end{equation}
\end{thm}

\begin{proof}
    First note that \(\E{X_{t}} = \nu_{\xi}^{t}\), by Markov's inequality \(\p{X_{t} > 0} \le
    \nu_{\xi}^{t}\). Thus it suffices to prove a lower bound in all cases discussed below.

    We first prove \eqref{IFWMJ}.
    If \(\ee{\xi \log \xi} < \infty\), then it follows from Theorem I.11.1 in \cite{athreya1972}
    that
    \begin{equation}\label{LMDTV}
        \lim_{t\to \infty} 
        \frac{
            \p{X_{t} > 0}
        }{
            \nu_{\xi}^{t}
        }
        = c > 0,
    \end{equation}
    for some constant \(c\). From this \eqref{IFWMJ} follows immediately.

    If \(\ee{\xi \log \xi} = \infty\), fix a large \(M>0\), consider the distribution \(\bar{\xi}=\min\{\xi,M\}\) and let
    \( (\bar{X}_{t})_{t \ge 0}\) be the branching process with offspring distribution \(\bar{\xi}\).
    Since \(\bar{\xi}\) is bounded, it follows from \eqref{LMDTV} that
    \begin{equation}\label{HQRAA}
        \liminf_{t \ge 0} \p{X_{t} > 0}^{1/t}
        \ge 
        \liminf_{t \ge 0} \p{\bar{X}_{t} > 0}^{1/t}
        =
        \ee{\bar{\xi}}
        \to \nu_{\xi}
        .
    \end{equation}
    as \(M \to \infty\). This proves the lower bound in \eqref{IFWMJ}.

    For \eqref{PAPPP}, first assume that \(\ee{\xi^{2}} < \infty\). In this case, Theorem
    1 of \cite{pakes1971} states that
    \begin{equation}\label{VNTMA}
        \frac{(Y_{t} \given X_{t} > 0)}{t}
        \to
        c_{\xi}
        \coloneqq
        1+
        \frac{\ee{\xi(\xi-1)}}{\nu_{\xi}(1-\nu_\xi)}
        ,
    \end{equation}
    in probability.
    Therefore
    \begin{equation}\label{JVQPR}
        \p{[Y_{t} < \omega(t)] \cap [X_{t} > 0]}
        =
        \p{Y_{t} < \omega(t) \cond X_{t} > 0}
        \p{X_{t} > 0}
        =
        (1+o(1)) ( (1+o(1)) \nu_\xi)^{t}
        ,
    \end{equation}
    from which \eqref{PAPPP} follows immediately.

    In the case that \(\ee{\xi^{2}} = \infty\), we again use a truncation argument. Let \(\bar{\xi}\) 
    and \( (\bar{X_{t}})_{t \ge 0}\) be as above. Let \(\bar{Y_{t}}=\sum_{i=1}^{t} \bar{X}_{i}\).
    Let \(A_{s}\) be the
    event that the first \(s\) nodes in BFS order in \( (X_{t})_{t \ge 0}\) have degree at most
    \(M\). Then
    \begin{equation}\label{HXDAU}
        \begin{aligned}
            \p{[Y_{t} < \omega(t)] \cap [X_{t} > 0]}
            &
            \ge
            \pp{[Y_{t} < 2 c_{\bar{\xi}}t] \cap [X_{t} > 0] \given A_{2 c_{\bar{\xi}}t}}
            \pp{A_{2 c_{\bar{\xi}}t}}
        \end{aligned}
    \end{equation}
    By Markov inequality, \(\p{\xi \ge M} \le \nu_\xi/M \leq  1/M\). Indeed, it is well-known that
    \(\ee{\xi} < \infty\) actually implies that \(\p{\xi \ge M} = o(1/M)\), where the asymptotics is
    as \(M \to \infty\). Since \(\bar{\xi} \leq M\), 
    \begin{equation}\label{AUYEN}
        c_{\bar{\xi}} = O(\ee{\bar{\xi}^{2}}) = O(M \ee{\bar{\xi}}) = O(M).
    \end{equation}
    Therefore
    \begin{equation}\label{GWBML}
        \pp{A_{2 c_{\bar{\xi}}t}}
        \ge
        (1-\p{\xi \ge M})^{2 c_{\bar{\xi}} t}
        \ge
        (1-o(1/M))^{O(M t)}
        \ge
        (1-\delta)^{t}
        ,
    \end{equation}
    for all \(\delta > 0\), provided that \(M\) is large enough. Also by choosing \(M\) large enough, we can get
    \(\ee{\bar{\xi}} > (1-\delta/2)\nu_\xi\) for all \(\delta > 0\).
    As \(\ee{\bar{\xi}^2}<\infty\), it follows from  \eqref{JVQPR} that
    \begin{equation}\label{PFMUB}
        \begin{aligned}
        \pp{[Y_{t} < 2 c_{\bar{\xi}}t] \cap [X_{t} > 0] \given A_{2 c_{\bar{\xi}}t}}
        &
        =
        \pp{[\bar{Y}_{t} < 2 c_{\bar{\xi}}t] \cap [\bar{X}_{t} > 0]}
        \\
        &
        =
        (1+o(1)) ( (1+o(1)) \ee{\bar{\xi}})^{t}
        \\
        &
        \ge
        (1-\delta) ((1-\delta)\nu_\xi)^{t}
        ,
        \end{aligned}
    \end{equation}
    for all \(t\) large enough.
    Putting \eqref{GWBML} and \eqref{PFMUB} into \eqref{HXDAU}, we have
    \begin{equation}\label{DUFVP}
        \liminf_{t \to \infty} 
        \p{[Y_{t} < \omega(t)] \cap [X_{t} > 0]}^{1/t}
        \ge
        (1-\delta)^{2} \nu_\xi.
    \end{equation}
    Since \(\delta\) is arbitrary, we are done with \eqref{PAPPP}.
\end{proof}

\section{Size-biased distributions}
\label{sec:size:bias}

We will apply the results on branching processes obtained in~\autoref{sec:branching} to distributions arising from $D$.
Define the \emph{in- and out-size biased} distributions of
\(D_{n}\), denoted by \( \dnin\) and \( \dnout\) respectively, by
\begin{equation}
    \p{\dnin=(k-1,\ell)}=\frac{k n_{k,\ell}}{m_{n}},
    \qquad
    \p{\dnout=(k,\ell-1)}=\frac{\ell n_{k,\ell}}{m_{n}}.
\end{equation}
In other words, if we choose a head $e^-$ uniformly at random (say $v$ is its incident node) and look at the number of heads/tails incident to $v$ different from $e^-$, what we get is a random pair of integers distributed as \( \dnin\).
Similarly, we get \( \dnout\) choosing a tail uniformly at random.

We also define the \emph{in- and out-size biased} distributions of \(D\), denoted by \(D_{\tin}\) and \(D_{\tout}\) respectively, by
\begin{equation}
\p{{D}_{\tin}=(k-1,\ell)} = \frac{k \lambda_{k,\ell}}{\lambda}
,
\qquad
\p{{D}_{\tout}=(k,\ell-1)} = \frac{\ell \lambda_{k,\ell}}{\lambda}
.
\end{equation}
Then, by (i) of \autoref{cond:main}, 
\( \dnin \to \din \) and \(\dnout \to \dout\) in distribution,
and by (iii) of \autoref{cond:main},
\begin{equation}\label{eq:dnin:nu}
    \lim_{n \to \infty}
    \E{
        \dnin^{+}
    }
    =
    \lim_{n \to \infty}
    \E{
        \dnout^{-}
    }
    =
    \E{\dinp}
    =
    \E{\doutm}
    =    
    \frac{\mathbb{E}\left[D^+ D^-\right]}
    \lambda
    =
    \nu
    .
\end{equation}

Alternatively, one can define the size-biased distributions of $D$ using generating functions. The bivariate probability generating function of \(D\) is 
\begin{equation}
f(z,w)\coloneqq \sum_{k,\ell} z^k w^\ell \lambda_{k,\ell}.
\end{equation}
The distributions \(\din\) and \(\dout\) have bivariate probability generating functions respectively
\begin{equation}
    f_{\tin}(z,w) = \frac{1}{\lambda} \frac{\partial f}{\partial z} 
    ,
    \qquad
    f_{\tout}(z,w) = \frac{1}{\lambda} \frac{\partial f}{\partial w} 
    .
\end{equation}
Note that \({\pdv{f}{z}} (1,1)= \pdv{f}{w} (1,1) =\lambda\), so \(f_{\text{in}}(1,1)=f_{\text{out}}(1,1)=1\).
(This shows that \(\din\) and \(\dout\) are indeed probability distributions.)
Similarly, the probability generating functions of \(\dinp\) and \(\doutm\) are \(f_\tin (1,w)\) and \(f_\tout (z,1)\) respectively.

We define
\begin{equation}
g(z,w)\coloneqq
\frac{\partial^2 f}{\lambda\,\partial z\partial w}
.
\end{equation}
Then
\begin{align}\label{eq:supercri:1}
    g(1,1) = \pdv{f_{\tin}}{w}(1,1) = \pdv{f_{\tout}}{z}(1,1) =  \dsE[\dinp] = \dsE[\doutm] = \nu.
\end{align}

Since \(\nu > 1\), there exists a unique solution \(\rho_\tin\in [0,1)\) of
\(f_\tin(1,w)=w\).  Similarly, there is a unique solution \(\rho_{\tout} \in [0,1)\) of
\(f_{\tout}(z,1)=z\).

The following is classical from branching process theory (see, e.g., Corollary 4.2 in
\cite{bordenave2016} or Theorem 3.1 in \cite{vanderhofstad2020}):
\begin{thm}\label{thm:GW:limit}
    Assume \autoref{cond:main} and \(\nu>1\). Let \(\rho_\tin,\rho_\tout \in [0,1)\) be the unique roots of
    \(f_\tin(1,w)=w\) and \(f_\tout (z,1)=z\), respectively.
Then \(s_{+} \coloneqq 1-\rho_\tin\) and \(s_{-} \coloneqq 1-\rho_\tout\) are the survival
probabilities of the branching processes with distribution \(D_\tin^+\) and \(D_\tout^-\)
respectively.
\end{thm}

{Recall the definitions $\hnu_+=g(1,1-s_+)$ and $\hnu_-=g(1-s_-,1)$ given in~\eqref{eq:hnu}. Let \(\hat{D}_{\tin}^{+}\) and \(\hat{D}_{\tout}^{-}\) be the conjugate distributions of \(\dinp\)
and \(\doutm\) respectively, defined as in~\eqref{eq:conj}. It is easy to check that 
\( \ee{\hat{D}_{\tin}^{+}}=\hnu_{+} \)
and
\( \ee{\hat{D}_{\tout}^{-}}=\hnu_{-}\). Note that \eqref{PXNEC} implies that \(\hnu_{+}, \hnu_{+} \in [0,1)\).}

\begin{rem}
While \(  \E{\dinp}= \E{\doutm}= \nu\), in general, \(\hnu_{+}\) and \(\hnu_{-}\)  are different. As
an example, fix an integer $\lambda\geq 2$,  let \(D^+\) be constant \(\lambda\) and let
\(D^-\) have a Poisson distribution with expectation \(\lambda\). Then \(\nu=\lambda\), \(D^+_\tin=
D^+\)  and \(D^-_\tout= D^-\). Thus, \(s_{+}=1\) and \(s_{-}<1\), so \(\hnu_{+}=0\) and
\(\hnu_{-}>0\). 

\end{rem}

\section{Exploring the graph}\label{sec:coupl}

We will explore \(\vecGn\) following a Breadth First Search (BFS) order. This technique to explore
vertex-neighbourhoods of a vertex is standard in the study of random graphs, see, e.g.,
\cite[Chapter~4]{vanderhofstad2016}.  However, in this paper it will be more convenient to study the
edge-neighbourhoods of a half-edge.  In this section we describe only the out-neighbourhoods of tails
since the study of in-neighbourhoods of heads is identical with only the exploration direction
reversed.

\subsection{The exploration process conditioning on a partial pairing}
\label{sec:explore}

For a set of nodes \(\cI\), let \(\cE^{\pm}(\cI)\) denote the set of heads and tails incident to
the nodes in \(\cI\). When \(\cI = \{v\}\), we also use \(\cE^{\pm}(v)=\cE^{\pm}(\cI)\). Let
\(\cE^{\pm}\coloneqq\cE^{\pm}([n])\) denote the set of all heads and tails respectively. For \(\cX\subseteq \cE^\pm\), let \(\cV(\cX)\) denote the set of vertices incident to $\cX$. When $\cX=\{e\}$, we also use \(v(e)\) to denote the only element of \(\cV(\cX)\).

Let \(H\) be a partial pairing of \(\cE^{\pm}\).
Let \(\cP^{\pm}(H)\subseteq \cE^{\pm}\) be the set of heads and tails that have been paired in \(H\) and write \(\cV(H)=\cV(\cP^\pm(H))\) for the set of nodes incident to \(\cP^{\pm}(H)\).  Let \(\cF^{\pm}(H)\coloneqq\cE^{\pm}(\cV(H))\setminus\cP^{\pm}(H)\) be the unpaired
heads and tails that are incident to \(\cV(H)\).  Let \(E_{H}\) denote the event that
\(H\) is part of the final half-edge pairing in \(\vecGn\).  We will explore the graph conditioning
on \(E_{H}\).

The exploration starts from an arbitrary tail \(e^{+} \in \cE^+\setminus \cE^+(\cV(H))\). Let \(v_{0}=v(e^+)\).  In this process, we create random pairings of half-edges one by one and keep
each half-edge in exactly one of the four states --- \emph{active, paired}, \emph{fatal} or \emph{undiscovered}.

More precisely, let \(\cA_i^{\pm }\), \(\cP_i^{\pm}\), \(\cF_{i}^{\pm}\) and  \(\cU_i^{\pm }\) denote
the set of heads and tails in the four states respectively after the \(i\)-th pairing of
half-edges. Initially we have 
\begin{equation}\label{eq:APUF}
    \cA_{0}^{+}=\{e^{+}\},\quad
    \cA_{0}^{-}=\cE^{-}(v_{0}) ,\quad
    \cP_0^{\pm}=\cP^{\pm}(H),\quad
    \cF_0^{\pm}=\cF^{\pm}(H),\quad
    \cU_{0}^{\pm}=\cE^{\pm} \setminus (\cA^{\pm}_{0} \cup    \cP^{\pm}_{0} \cup \cF_0^{\pm})    .
\end{equation}
Then we set \(i=1\) and run the following procedure: \leavevmode
\begin{enumerate}[\normalfont(i)]
    \item Let \(e_{i}^{+}\) be the tail which became active earliest in
        \(\cA_{i-1}^{+}\). (If multiple such tails exist, choose an arbitrary
        one among them. Note that \(e_{{1}}^{+}=e^{+}\).) 
    \item Pair \(e^{+}_{i}\) with a head \(e^{-}_{i}\) chosen uniformly at random from
        \(\cE^{-}\setminus \cP_{i-1}^{-}\), i.e., from all unpaired heads. Let
        \(\cP_{i}^{\pm} = \cP_{i-1}^{\pm} \cup \{e^{\pm}_{i}\}\).
    \item If \(e^{-}_{i} \in \cF_{i-1}^{-}\), then terminate; if \(e^{-}_{i} \in \cA_{i-1}^{-}\),
        then \(\cA_{i}^{\pm} = \cA_{i-1}^{\pm}\setminus \{e_i^\pm\}\); and if $e_i^-\in
        \cU^-_{i-1}$, then \(\cA_{i}^{\pm} = (\cA_{i-1}^{\pm} \cup \cE^{\pm}(v_{i})) \setminus
        \{e^{\pm}_{i}\}\)  where \(v_{i}=v(e^-_i)\).
        \item If \(\cA_{i}^{+}=\emptyset\), then terminate; otherwise set \(\cF_{i}^{\pm} = \cF_{i-1}^{\pm}\), $\cU_i^{\pm} =\cE^{\pm}\setminus (\cA^{\pm}_{i} \cup    \cP^{\pm}_{i} \cup \cF_i^{\pm})$, \(i=i+1\) and go to (i).
\end{enumerate}
In words, the exploration process exposes edge by edge of $\vecGn$ in a BFS order and stops either
when it hits $\cV(H)$ or when all tails that can be reached from $e^+$ have been paired.

In parallel to the exploration process, we construct a sequence of rooted trees
\(T_{e^+}(i)\){, whose nodes represent tails in $\cE^+$}. Let \(T_{e^+}(0)\) be a tree with a
single node corresponding to \(e^{+}\). We construct \(T_{e^+}(i)\) as follows:  if \(e_{i}^{-} \in
\cU_{i-1}^{-}\), then construct \(T_{e^+}(i)\) from  \(T_{e^+}(i-1)\) by adding
\(\abs{\cE^{+}(v_{i})}\) child nodes to the node representing \(e_{i}^{+}\), each one representing a
tail in \(\cE^{+}(v_{i})\); otherwise, let \(T_{e^+}(i)=T_{e^+}(i-1)\). See \autoref{fig:explore}
for an example of the exploration process and the corresponding tree.

Given half-edges $e_1,e_2$, we define the distance from $e_1$ to $e_2$, denoted by
$\dist(e_1,e_2)$, as the graph distance from $v(e_1)$ to $v(e_2)$ in \(\vecGn\).
For example, in \autoref{fig:explore},
\(\dist(e_{1}^{+}, e_{5}^{+})=2\) and \(\dist(e_{2}^{+}, e_{2}^{-})=1\).

\begin{obs}\label{rem:distances}
    Note that the tree $T_{e^+}(i)$ preserves
    distances of tails in \(\vecGn\): if a node corresponding to a tail is at distance \(t\) from the root,
    then the tail is at distance \(t\) from \(e^{+}\) in \(\vecGn\). Therefore, the number of
    nodes in the \(t\)-th level of the tree is the number of tails in \(\cE^+\) at
    distance \(t\) from $e^+$.
\end{obs}

While $T_{e^+}(i)$ is an unlabelled tree, its set of nodes corresponds to the set of tails $\cP^+_i\cup\cA^+_i$. Therefore, we can assign a label \emph{paired} or \emph{active} to each node. 

\begin{figure}[ht]
\centering
\begin{subfigure}{.5\textwidth}
  \centering
  \includegraphics[width=0.8\linewidth]{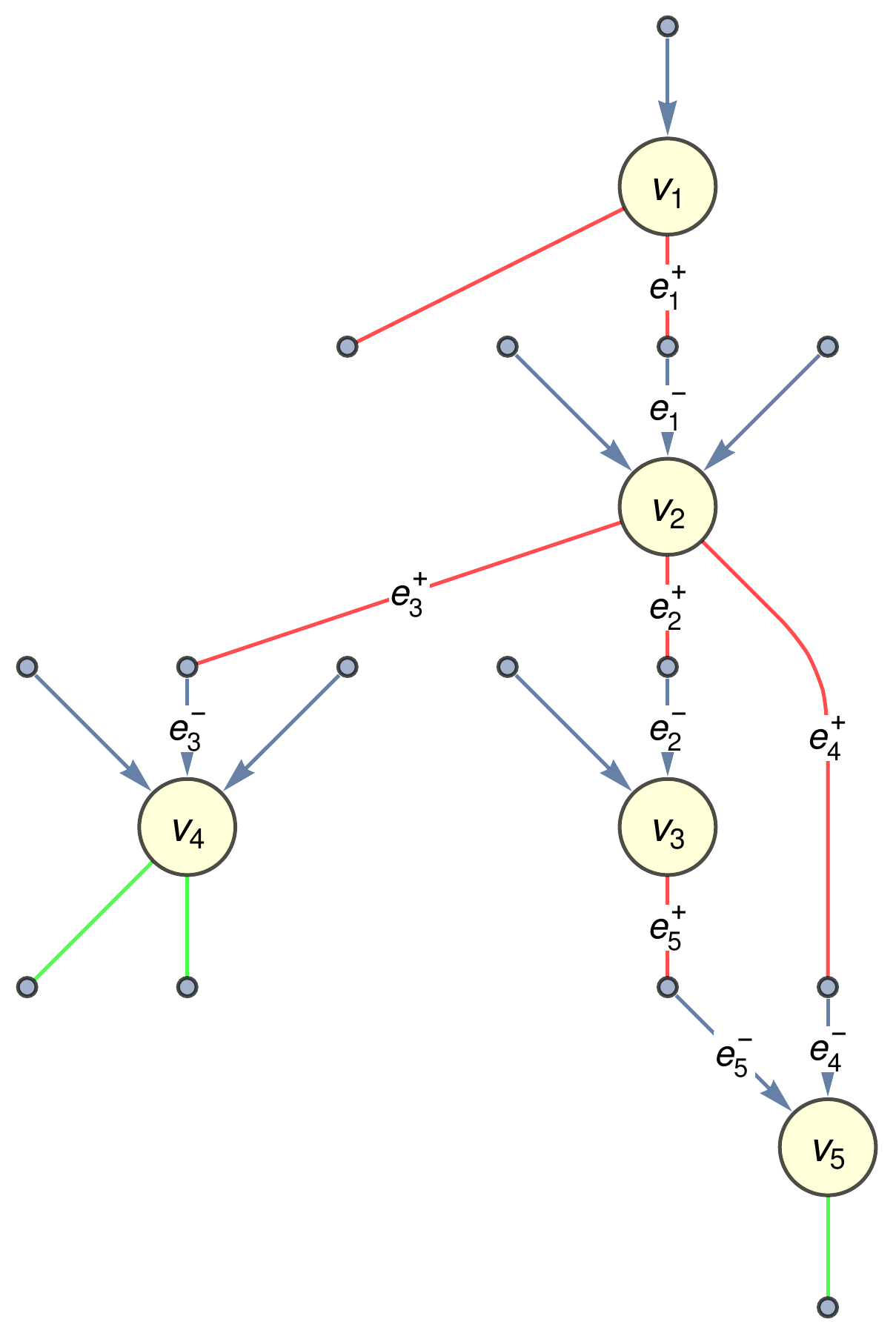}
\end{subfigure}%
\hfill
\begin{subfigure}{.5\textwidth}
  \centering
  \includegraphics[width=0.9\linewidth]{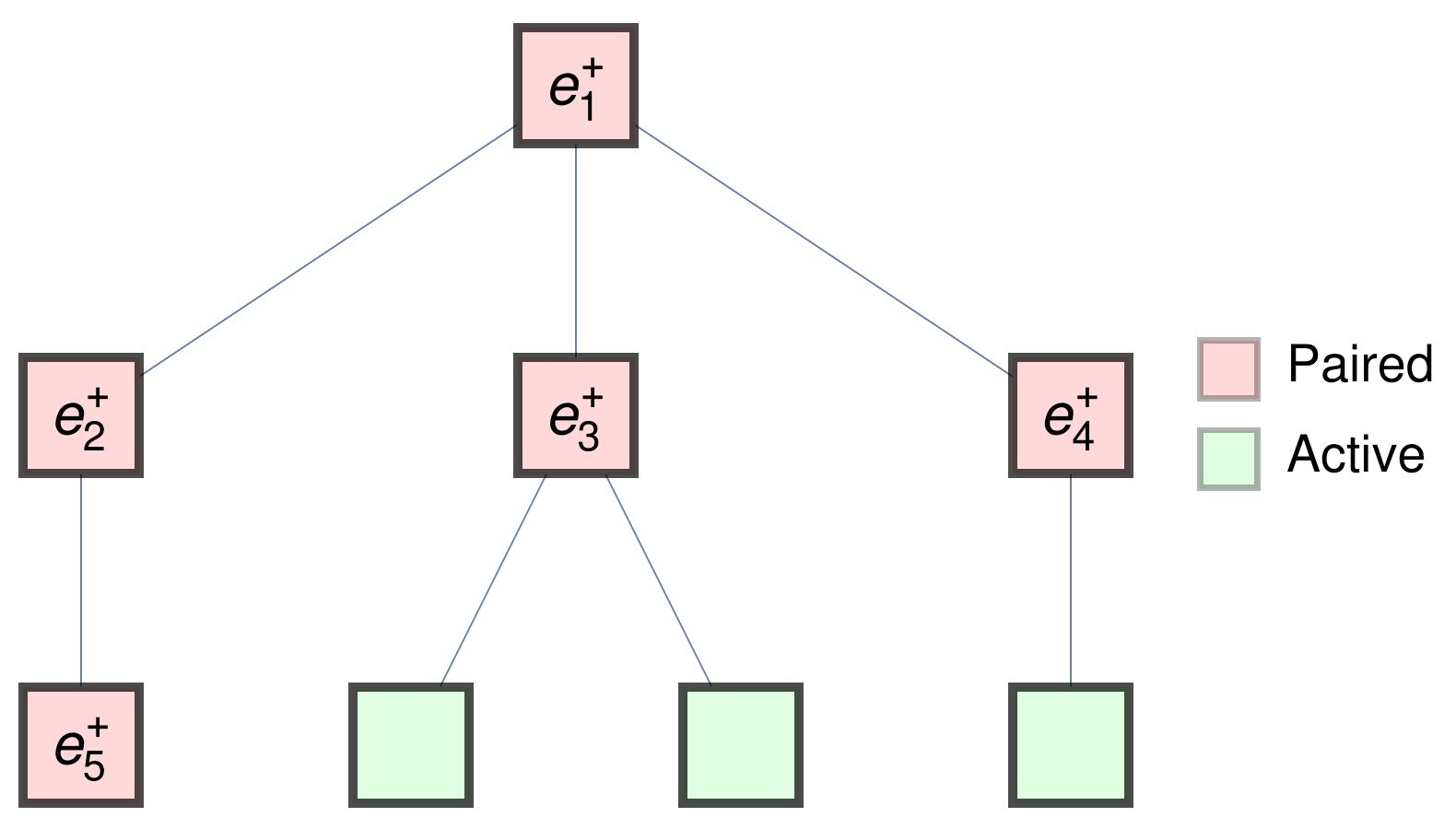}
\end{subfigure}
\caption{An ongoing exploration process and its associated tree}
\label{fig:explore}
\end{figure}

We split the exploration process into epochs. At the $t$-th epoch, we pair all the tails at distance $t$ from $e^+$. Let $i_t$ be the last step of epoch $t$. Then, $T_{e^+}(i_t)$ has the following properties: (i) has depth $t$; (ii) all nodes in the $t$-th level are active; and (iii) all nodes in the $j$-th level for $j<t$ are paired.

We call a rooted tree $T$ \emph{incomplete} if it satisfies (i)-(iii) (in the sense that the
subtrees rooted at the last level have not been decided yet). We let \(p(T)\) be the number of \emph{paired} nodes in \(T\). 

\subsection{Coupling the exploration and branching processes}

Let \(Q_n\coloneqq(D_n)_\tin^+\) be the distribution obtained by taking the marginal on the second component of the in-size biased distribution of \(D_n\); i.e., for all \(\ell\geq 0\),
\begin{equation}
    \p{Q_n=\ell} = q_{n,\ell} \coloneqq \frac{\sum_{k\geq 1}k n_{k,\ell}}{m_n}.
\end{equation}

Recall that in \autoref{sec:size:bias} it has been shown that \(Q_{n} \to \dinp\) in distribution
and in expectation. In particular, by \eqref{eq:dnin:nu} \(\bE[Q_n] \to \bE[\dinp] = \nu\).  Let
\(\hat{Q}_{n}\) be the conjugate of \(Q_{n}\).  It follows from \autoref{lem:conj} that \(\hat Q_n
\inlaw \dinpconj\) and \(\e[\hat Q_{n}] \to \bE[\hat D_\tin^+]=\hnu_{+}\).  Hence,
\(\bE[Q_n]= (1+o(1))\nu\) and \(\bE[ \hat Q_n]= (1+o(1))\hnu_{+}\). 

In order to transfer the results from branching processes to the graph exploration process, we need to
introduce two new probability distributions \(Q_n^{\downarrow}\) and \(Q_n^{\uparrow}\) by slightly
perturbing \(Q_n\).

{For \(\beta\in(0,1/10)\)} consider the
probability distribution
\(Q_n^\da=Q_n^\da (\beta)\) defined by
\begin{align*}
    \p{\Qnda = \ell}
    =
    q_{n,\ell}^\da \coloneqq
    \begin{cases}
        c^{\da} q_{n,\ell} & \text{if } q_{n,\ell} \geq n^{-2\beta} \text{ and }\ell\leq
        n^{\beta}
        \\
        0 & \text{otherwise}
    \end{cases}
\end{align*} 
where \(c^{\da}\) is a normalising constant. It is easy to check that \autoref{cond:main} implies that
\(c^{\da}=1+{o}(n^{-\beta})\) and
\begin{equation}
    \sum_{\ell> n^{\beta}} \ell q_{n,\ell}= o(1),
    \qquad
    \sum_{\ell: q_{n,\ell}<n^{-2\beta}} \ell q_{n,\ell} =  o(1).
\end{equation}
Therefore, \( \bE[Q_n^\da] \to \nu.  \)

Similarly, the probability distribution \(\Qnua = \Qnua(\beta)\) is defined by
\begin{equation}\label{VMLKM}
    \p{\Qnua = \ell}
    =
    q_{n,\ell}^\ua \coloneqq
    \begin{cases}
        c^{\ua} q_{n,\ell} & \ell \ge 1
        \\
        c^{\ua} q_{n,0}+n^{-1/2+2\beta} & \ell = 0
    \end{cases}
\end{equation}
where \(c^{\ua} = 1+O(n^{-1/2+2\beta})\) is a normalising constant. Again, by \autoref{cond:main},
\( \bE[Q_n^\ua] \to \nu.  \)

\begin{obs}\label{obs:cont}
    Let \(\hQnda\) be the conjugate distribution of \(\Qnda\).
    It follows from \autoref{lem:conj} that \(\hQnda \inlaw \dinpconj\) 
    and \(\e[\hat Q^{\da}_{n}] \to \bE[\hat D_\tin^+]=\hnu_{+}\). 
    So we can apply \autoref{thm:BP} to a
    branching process with offspring distribution \(Q_n^\da\)  by taking \(\hnu_{\xi}=(1+o(1))\hnu_{+}\). 
    The same applies to \(\Qnua\) and \(\hQnua\).
\end{obs}

Let \(\GW_{\xi}\) be a Galton-Watson tree with offspring distribution \(\xi\).  For an incomplete rooted tree $T$ of depth $t$, we use the notation
\(\GW_{\xi} \cong T\) to denote that \(T\) is a root subtree of \(\GW_{\xi}\) and all paired nodes of \(T\) have the same degree in \(\GW_{\xi}\). 

\begin{lemma}\label{lem:eq_T_GW}
    Let $\beta \in (0,1/10)$ and let \(H\) be a partial pairing with \(|\cV(H)|\le n^{1- 6\beta}\).
    For every incomplete tree \(T\) with
    \(p(T)\leq n^{\beta}\), we have
    \begin{equation}
        (1+o(1)) \p{\GW_{Q_n^\da{(\beta)}}\cong T} 
        \leq \p{T_{e^{+}}(p(T)) = T \cond E_{H}} 
        \leq  (1+o(1)) \p{\GW_{Q_n^{\ua}(\beta)}\cong T}
        ,
    \end{equation}
    where the implicit functions \(o(1)\) are uniform over all such \(T\) and \(H\).
\end{lemma}

\begin{proof}
    We start with the upper bound.  First, by \autoref{lem:s}, \(|\cV(H)|\le n^{1- 6 \beta}\)
    implies that
    \begin{equation}
        \abs{\cP^{-}(H)}
        \le
        d_{\cV(H)}(1,0)
        =
        o(n^{1-{3\beta}})
        .
    \end{equation}
    Let \(X_{i} = \abs{\cA_{i}^{{+}}}-\abs{\cA_{i-1}^{{+}}} +1\) be the number of tails that became active during the \(i\)-th step of the process. 
    Let \(\ell_i\) be the number of children of the \(i\)-th node in \(T\) in the
    BFS order.  Let \(E_{i} = \cap_{j = 1}^{i} [X_{j} = \ell_{j}]\).
    Let \(q_{n,\ell}(i)\coloneqq\p{X_i=\ell\mid E_{i-1}}\).
    Then for all \(\ell \ge 1\) and \(i \le p(T)\),
    \begin{align}\label{eq:rel_i_pnl}
        q_{n,\ell}(i) 
        \leq\frac{\sum_{k\geq 1} k\, n_{k,\ell}}{m_n - (i-1)-\abs{\cP^{-}(H)}} 
        =  (1+o(n^{-3\beta})) q_{n,\ell}^{} 
        =  (1+o(n^{-3\beta})) q_{n,\ell}^{\ua}
        . 
    \end{align}
    Recall that by \autoref{cor:D} we have \(\Delta_{n}=o(\sqrt{n})\).
    Since at most 
    \(p(T) \Delta_{n} = o(n^{1/2+\beta})\) heads are active, we also have
    \begin{align}\label{eq:rel_i_pn0}
        q_{n,0}(i) 
        \leq
        \frac{p(T) \Delta_{n}  + \sum_{k\geq 1} k\, n_{k,0}}{m_n - (i-1)-\abs{\cP^{-}(H)}} 
        =
        o(n^{-1/2+\beta})
        +
        (1+o(n^{-{3\beta}}))q_{n,0}^{} 
        \le
        (1+o(n^{-{\beta}}))q_{n,0}^{\ua} 
        ,
    \end{align}
    where the last step uses that \(q_{n,0}^{\ua} \geq  n^{-1/2+{2\beta}}\).
    It follows that
    \begin{equation}
        \p{T_{e^{+}}(p(T)) = T \cond E_{H}} 
        = \prod_{i=1}^{p(T)} q_{n,\ell_{i}}(i)
        \leq \prod_{i=1}^{p(T)} (1+o(n^{-\beta})) q_{n,\ell_i}^{\ua}
        = (1+o(1))\p{\GW_{Q_{n}^{\ua}(\beta)}\cong T}.
    \end{equation}

   We now prove the lower bound. We first show that \(q_{n,\ell}(i)\geq
    (1+o(n^{-\beta}))q^\da_{n,\ell}\). We may assume that \(\ell\leq n^{\beta}\) and
    \(q_{n,\ell}\geq n^{-2\beta}\), as otherwise \(q^\da_{n,\ell}=0\) and the claim holds trivially.  Since
    there are at most \(p(T) \Delta_{n} + d_{\cV(H)}(1,0) = o(n^{1-3\beta})\) heads in $\cA_{i-1}^{-} \cup \cP_{i-1}^{-}{\cup \cF_{i-1}^-}$ and at most \(m_n\) heads that
    have not been paired yet, we have for \(\ell \ge 0\),
    \begin{equation}
        q_{n,\ell}(i) \geq \frac{\sum_{k\geq 1} k n_{k,\ell} - p(T) \Delta_{n} - d_{\cV(H)}(1,0) }{m_n}
        = 
        q_{n,\ell}-o(n^{-{3\beta}})
        =
        (1+o(n^{-{\beta}})) q_{n,\ell}
        =
        (1+o(n^{-\beta}))q_{n,l}^{\da} ,
    \end{equation}
    where the last step uses that $q_{n,\ell}=(1+o(n^{-\beta}))q_{n,\ell}^\da$.    The rest of the
    argument is analogous to the upper bound.
\end{proof}

\section{Thin neighbourhoods}\label{sec:tower}

Given $e_1^\pm,e_2^\pm\in \cE^\pm$, recall that $\dist(e_1^\pm,e_2^\pm)$ is the distance between the nodes incident to them
in \(\vecGn\). Let \(\cN^\pm_{t}(e^{\pm})\) be the set of heads/tails at distance
\(t\) from a head/tail \(e^{\pm}\) in \(\vecGn\).  We will only use tail-neighbourhoods of tails or head-neighbourhoods of heads. Thus, we talk about the edge-neighbourhood of a tail/head and denote it by \(\cN_{t}(e^{\pm})\).

Throughout the rest of the paper, we fix 
\begin{equation}
\omega\coloneqq \log^6 n.
\end{equation}

\subsection{Supercritical case}
In this subsection, we assume that \(\nu > 1\). Let \(t_{\omega}(e^{\pm})\) be the first time that the edge-neighbourhood of \(e^{\pm}\) has size
at least \(\omega\), i.e., 
\begin{equation}
    t_{\omega}(e^{\pm})
    \coloneqq
    \inf
    \left\{ 
        t \ge 1:
        \abs{
            \cN_{t}(e^{\pm})
        }
        \geq \omega
    \right\}
    .
\end{equation}
We call \(t_{\omega}(e^{\pm})\) the \emph{expansion time} of $e^\pm$. If such expansion never
happens, let \(t_{\omega}(e^{\pm}) = \infty\).

When \(\hnu_{\pm}>0\) and \(\nu > 1\), 
we are interested in expansions which happen at around time
\begin{equation}\label{EOSJR}
    t^{\pm} \coloneqq \frac{\log(n)}{\log(1/\hnu_{\pm})}
    .
\end{equation}
More specifically, we will show that whp there is no expansion after time \( (1+\delta)t^{\pm}\) and
there exist out/in explorations expanding after time \( (1-\delta)t^{\pm}\), producing atypically thin neighbourhoods.

When \(\hnu_{\pm} = 0\), we will show that whp, the expansion of out/in-neighbourhood
happens before time \(\delta \log(n)\) for all \(\delta > 0\).

The proof relies on \autoref{lem:eq_T_GW}, which allows approximating the probability of finding a
thin neighbourhood with the probability of the corresponding event in a branching process.
\begin{prop}\label{prop:towers2}
    Assume that \(\nu > 1\).
    Let $\gamma>0$ and \(e^{\pm}\in \cE^{\pm}\).  Let
    \begin{equation}
        A(e^{\pm}, t)
        =
        \left[
        \cap_{r=1}^{t}
        [
           0< \abs{\cN_{r}(e^{\pm})}<\omega
		]        
        \right]
    \end{equation}
    Then uniformly for every partial pairing \(H\) with \(\abs{\cV(H)}
    \le n^{1-\gamma}\) and every \(t = \Theta(\log n)\), we have
    \begin{equation}\label{eq:tower:p}
        \p{
            A(e^{\pm}, t)
        \cond E_{H}}
        =
        \begin{cases}
            \hnu_{\pm}^{(1+o(1))t}
            &
            (\hnu_{\pm}>0)
            \\
            O(\zeta^{t})
            &
            (\hnu_{\pm}=0)
        \end{cases}
    \end{equation}
    for all \(\zeta>0\).
\end{prop}
\begin{proof}
    We only prove the upper bounds for \(A(e^{+},t)\) in \eqref{eq:tower:p}.  The lower bound follows from similar
    arguments.  

    Let \(\cT_{t,\omega}\) be the class of incomplete trees of depth \(t\) where each level has less than \(\omega\) nodes. 
    For \(T \in \cT_{t, \omega}\), we have \(t-1\leq p(T) \le \abs{T} \le (\omega-1) t =
    o(n^{\gamma/6})\). 
    Let \(X_r^{{\ua}}\) be the size of the
    \(r\)-th generation of a branching process (Galton-Watson tree) with offspring distribution
    \(Q_n^{\ua}\). 
    Let \(\hQnua\) be the conjugate of \(\Qnua\). Then by the construction of \(\Qnua\),
    we have \(\e[\hQnua] >0\) and \(\e[\hQnua] \to \hnu_{+}\). 
    It follows from \autoref{rem:distances}, \autoref{obs:cont} and
    \autoref{lem:eq_T_GW} with $\beta = \gamma/6$ that the left-hand-side of \eqref{eq:tower:p} is
    \begin{align}
        \sum_{i=t-1}^{\floor{(\omega-1) t}} \sum_{\substack{T\in \cT_{t,\omega}\\ p(T)=i}} 
        \p{
            T_{e^{+}}(i) = T
             \given
             E_{H}
        }
        & \leq   (1+o(1))\sum_{i=t-1}^{\floor{(\omega-1) t}} 
        \sum_{\substack{T\in \cT_{t,\omega}\\ p(T)=i}} \p{\GW_{Q_n^{\ua}(\beta)}\cong T}
        \\
        &
        =
        (1+o(1))
        \p{\cap_{r=1}^{{t_{}}} [0<X_r^{{\ua}}<\omega]}
        .
    \end{align}
    By \autoref{rem:seneta} and since \(t=\Omega(\log n)\), we have that \(t_{\xi}(\omega)=o(t)\)
    with \(\xi=\Qnua\).  Then by \autoref{thm:BP}, the above is at most
    \begin{equation}
        (1+o(1)) (\e{[\hQnua]})^{(1+o(1))t}
        =
        \begin{cases}
            \hnu_{+}^{(1+o(1))t} & (\hnu_{+} > 0)
            \\
        O(\zeta^{t}) & (\hnu_{+} = 0)
        \end{cases}
    \end{equation}
    for all \(\zeta > 0\).
\end{proof}

The following lemma
shows that whp no expansion happens later than \( (1+\delta) t^{\pm}\), for $\delta>0$.

\begin{lemma}\label{lem:tower:2}
    Assume that \(\nu > 1\).
    Let $\delta\in (0,1)$ and  let
    \begin{equation}
        B_{1}(e^{\pm})
        =
        B_{1}(e^{\pm};\delta)
        =
        \begin{cases}
            A(e^{\pm}, (1+\delta)t^{\pm})
            &
            (\hnu_{\pm} > 0)
            \\
            A(e^{\pm}, \delta \log(n))
            &
            (\hnu_{\pm} = 0)
            \\
        \end{cases}
    \end{equation}
    Let \(B_{1} = \cup_{e^{\pm} \in \cE^{\pm}} B_{1}(e^{\pm})\).
    Then \(\p{B_{1}}=o(1)\).
\end{lemma}
\begin{proof}
By taking \(H\) to be an empty partial pairing,
it follows from Proposition~\ref{prop:towers2} that when \(\hnu_{\pm}>0\)
\begin{equation}
    \p{A(e^{\pm}, (1+\delta)t^{\pm})} \leq \hnu_{\pm}^{(1-\delta/2)\floor{(1+\delta)t^{\pm}}}
    \leq \hnu_{\pm}^{(1+\delta/4)t^{\pm}} \leq n^{-(1+\delta/4)},
\end{equation}
and when \(\hnu_{\pm}=0\)
\begin{equation}
    \p{A(e^{\pm}, \delta\log(n))} = O(\zeta^{\delta \log(n)})  \leq n^{-(1+\delta/4)},
\end{equation}
by choosing \(\zeta\) small enough with respect to \(\delta\).
The lemma follows from a union bound over all half-edges.
\end{proof}

To show the existence of late time expansions, we extend \autoref{prop:towers2} to the following:
\begin{lemma}\label{lem:tower:0}
    Assume that \(\nu > 1\).
    Let $\delta\in (0,1)$  and  \(e^{\pm}\in \cE^{\pm}\). 
    \begin{enumerate}[\normalfont(i)]
        \item 
            If \(\hnu_{\pm} \ge 0\), let
            \begin{equation}
                B_{2}(e^{\pm})
                =
                B_{2}(e^{\pm};\delta)
                =
                \left[ 
                    \abs{\frac{t_{\omega}(e^{\pm})}{t^{\pm}}-1} < \delta
                \right]
                \cap
                \left[ 
                    \abs{
                        \cN_{t_{\omega}(e^{\pm})}(e^{\pm})
                    }
                    < 
                    \omega^{2}
                \right]
                .
            \end{equation}
            Then uniformly for every partial paring \(H\) with \(\abs{\cV(H)}
            \le n^{1-\delta/{5}}\), we have
            \begin{equation}\label{TQBGG}
                n^{-1+\delta/2}\leq \p{
                    B_{2} (e^{\pm})\cond E_{H}
                }
                \leq 
                n^{-1+3\delta/2}
                .
            \end{equation}
        \item 
            If \(\hnu_{\pm} = 0\), let
            \begin{equation}
                B_{2}(e^{\pm})
                =
                B_{2}(e^{\pm};\delta)
                =
                \left[ 
                    t_{\omega}(e^{\pm}) < \delta \log(n)
                \right]
                \cap
                \left[ 
                    \abs{
                        \cN_{t_{\omega}(e^{\pm})}(e^{\pm})
                    }
                    < 
                    \omega^{2}
                \right]
                .
            \end{equation}
            Then uniformly for every partial paring \(H\) with \(\abs{\cV(H)}
            \le n^{1-\delta/{5}}\), we have
            \begin{equation}\label{BXWGG}
                \p{
                    B_{2} (e^{\pm})\cond E_{H}
                }
                =
                1-o(1)
                .
            \end{equation}
    \end{enumerate}
\end{lemma}

\begin{proof}[Proof of (i)]
    We only bound the probability of $B_2(e^+)$; the proof for $B_2(e^-)$ is analogous.
    Let \(t_{1} = \ceil{(1-\delta) t^{+}}\) and \(t_{2} = \floor{(1+\delta)t^{+}}\).
    Let 
    \begin{equation}\label{ATZTW}
        \begin{gathered}
            A_{1} = A(e^{+}, t_{1})
            ,
            \quad
            A_{2} = A(e^{+}, t_{2})
            ,
            \quad
            A_{3} = [\cN_{t_{2}}({e^{+}}) = \emptyset]
            ,
            \\
            A_{4} = 
            \left[
                t_{\omega}(e^+)<\infty
            \right]
            \cap
            \left[
                \abs{
                    \cN_{t_{\omega}(e^+)}(e^{+})
                }
                \geq 
                \omega^{2}
            \right]
            .
        \end{gathered}
    \end{equation}
Using \autoref{prop:towers2} with $\gamma=\delta/5$, we have
\begin{align}\label{IFXTT}
        \p{B_{2}(e^{+})
            \cond E_{H}
        }
        &
        \le
        \p{A_{1}
            \cond E_{H}
        }
        \le n^{-1 +3\delta/2}
        .
    \end{align}

    We now prove the lower bound.
    If \(A_{1}\) happens, then there are three cases in which \(B_{2}(e^{+})\) does not happen: (i) the
    neighbourhood of \(e^{+}\) survives but does not expand by time \(t_{2}\); (ii) the neighbourhood dies
    by time $t_2$; (iii) the neighbourhood expands by time $t_2$, but expands by too much.  Thus
    \begin{equation}\label{eq:3:A}
    \begin{aligned}
        &
        \p{B_{2}(e^{+})
            \cond E_{H}
        }
        \\
        &
        \ge
        \p{A_{1}
            \cond E_{H}
        }
        -
        \p{A_{2}
            \cond E_{H}
        } 
        - 
        \p{A_{1} \cap A_{3}
            \cond E_{H}
        }
        -
        \p{A_{1} \cap A_2^c\cap A_{4}
            \cond E_{H}
        }
        \\
        &
        =
        \p{A_{1}
            \cond E_{H}
        }\left(1- 
        \p{A_{3}
            \cond A_1\cap E_{H}
        }
        -
        \p{A_2^c\cap A_{4}
            \cond A_1\cap E_{H}
        }
        \right)
        -
        \p{A_{2}
            \cond E_{H}
        } 
        .
    \end{aligned}
    \end{equation}
    It follows directly from \autoref{prop:towers2} with $\gamma=\delta/5$ that
    \begin{equation}\label{eq:A:1:2}
        \p{A_{1}
            \cond E_{H}
        } \geq  n^{-1+3\delta/4},
        \qquad
        \p{A_{2}
            \cond E_{H}
        } \leq  n^{-1-3\delta/4}.
    \end{equation}
        
Let \(X_t^{{\ua}}\) be the size of the \(t\)-th generation of a branching process (Galton-Watson tree) with
offspring \(Q_n^{\ua}\).
Consider the analogue of the events $A_i$ for $X_t^\ua$:
\begin{align*}
A_1^*&=\big[\cap_{r=1}^{{t_{1}}} [0<X_r^{\ua}<\omega]\big]\\
A_2^*&=\big[\cap_{r=1}^{{t_{2}}} [0<X_r^{\ua}<\omega]\big]\\
A_3^*&=\big[X^{\ua}_{t_2}=0\big]\\
A_4^*&=\big[t_\omega<\infty, X^{\ua}_{t_\omega}\geq \omega^2\big]
\end{align*} 
where $t_\omega$ is the smallest $t$ such that $X_t^\ua\geq \omega$, or $t_\omega=\infty$ if it does not exist.
One can transfer the probability of any event for branching processes to the corresponding event in the graph exploration process conditional on $E_H$ in a similar way as in the proof of \autoref{prop:towers2}. Thus, it suffices to upper bound the remaining probabilities in~\eqref{eq:3:A} for the branching process analogues.
    
Since the survival probability of $X_t^\ua$ tends to \(s_{+}\), we have
    \begin{equation}\label{eq:A:1:3}
        \begin{aligned}
         \p{
                    A_3^*
                \cond
                A_1^*
            }
            &
            \leq 
            \p{X_{t_{2}-t_{1}+1} = 0}
            \leq
            1-s_{+} +o(1)
            .
        \end{aligned}
    \end{equation}
By Markov inequality, for any $t\geq t_1$,
    \begin{equation}
        \p{X_{t+1}^\ua\geq \omega^2 \cond A_1^*\cap [X_{t}^\ua < \omega]}
        \le
        \frac{
            \E{X^\ua_{t+1} \cond A_1^*\cap [X^\ua_{t} < \omega]}
        }{
            \omega^{2}
        }
        \le
        \frac{ (1+o(1)) \nu \omega}{\omega^{2}}
        =
        O(\omega^{-1})
        .
    \end{equation}
The event $A_1^*\cap (A_2^*)^c\cap A_4^*$ implies that there exists $t\in [t_1,t_2)$ such that $X_t^\ua<\omega$ and $X_{t+1}^\ua\geq \omega^2$. Therefore,
    \begin{equation}\label{eq:A:1:4}
		\p{(A_2^*)^c\cap A_4^*\cond A_1^*}    		
		 \leq     
        \sum_{t=t_1}^{t_2-1}\p{X^\ua_{t+1}\geq \omega^2 \cond A_1^*\cap [X^\ua_{t} < \omega]}
        =
        O(\omega^{-1}\log{n})
        =
        o(1)
        .
    \end{equation}
    Then part (i) of the lemma follows by transferring the probabilities to the original events
    conditional on $E_H$ and putting \eqref{eq:A:1:2}, \eqref{eq:A:1:3} and \eqref{eq:A:1:4} into
    \eqref{eq:3:A}.
\end{proof}

\begin{proof}[Proof of (ii)]
    Again we only bound the probability of $B_2(e^+)$.
    Let \({t_{2}} = \ceil{\delta \log n}\). 
    Let
    \begin{equation}\label{PJGQY}
        A_{1} = A(e^{+}, t_2)
        ,
        \quad
        A_{2} = [\cN_{t_2}({e^{+}}) = \emptyset]
        ,
        \quad
        A_{3} = 
        \left[
            t_{\omega}(e^+)<\infty
        \right]
        \cap
        \left[
            \abs{
                \cN_{t_{\omega}(e^+)}(e^{+})
            }
            \geq 
            \omega^{2}
        \right]
        .
    \end{equation}
    When \(B_{2}(e^{+})\) does not happen, there are three (non-exclusive) cases: (i) the
    neighbourhood of \(e^{+}\) survives till \(t_2\) but does not expand; (ii) the neighbourhood
    of \(e^{+}\) dies by time \(t_2\); (iii) the neighbourhood of \(e^{+}\) expands too much. Thus
    \begin{equation}\label{MLIOC}
        \p{B_{2}(e^{+})^{c} \cond E_{H}}
        \le
        \p{A_{1} \cond E_{H}} + \p{A_{2} \cond E_{H}} + \p{A_{3} \cond A_{1}^{c} \cap A_{2}^{c} \cap E_{H}}.
    \end{equation}
    Note that it follows from \autoref{prop:towers2} that \(\p{A_{1} \cond E_{H}} = o(1)\).

    Let \(X_t^\ua\) and $t_\omega$ be as in the proof of (i), where the conjugate is defined as in~\eqref{XHNOP}.  Consider the branching process analogue of the events
    $A_i$ for $X_t^\ua$:
    \begin{equation}\label{DCUBN}
        A_1^*=\left[\cap_{r=1}^{{t_2}} [0<X_r^\ua<\omega]\right],
        \quad
        A_2^*=\left[X_{t_2}^\ua=0\right],
        \quad
        A_3^*=\big[t_\omega<\infty, X_{t_\omega}^\ua\geq \omega^2 \big],
    \end{equation} 
    By the same argument as in (i), it suffices to compute the remaining probabilities
    in~\eqref{MLIOC} for the branching process analogues.

    Note that \(\hnu_{+} = 0\) implies \(s_{+}=1\). Thus
    \begin{equation}\label{TUPAS}
        \p{A_{2}^{*}} \le \p{(X_{t}^\ua)_{t \ge 0} \text{ extinguishes}}  \to  1-s_{+} = 0.
    \end{equation}
    When \((A_{1}^{*})^{c} \cap (A_{2}^{*})^{c}\) happens, there must be the expansion before time \(\delta
    \log n\).  Thus by an argument similar to that of \eqref{eq:A:1:4}, we also have
    \begin{equation}\label{LJAKI}
        \p{A_{3}^{*} \cond (A_{1}^{*})^{c} \cap (A_{2}^{*})^{c}}
        = O(\omega^{-1} \log n) = o(1).
    \end{equation}
    Then part (ii) of the lemma follows by transferring the probabilities to the original events
    conditional on $E_H$ and putting \eqref{TUPAS}, \eqref{LJAKI} into
    \eqref{MLIOC}.
\end{proof}

Next we show that whp there exist thin out-neighbourhoods and thin in-neighbourhoods of expected
height which do not intersect.

\begin{lemma}\label{lem:tower:1}
    Assume that \(\nu > 1\).
	Let \(B_{2}(e^{\pm})\) be as in \autoref{lem:tower:0}. Define
    \begin{enumerate}[\normalfont(i)]
	\item if  \(\hnu_{+}>0\) and \(\hnu_{-} >0\), for every $e^+\in \cE^+$ and $e^-\in \cE^-$
    \begin{equation}\label{CMJFQ}
        B_{3}(e^+,e^-)
        = 
        B_{3}(e^+,e^-;\delta)
        =
        B_{2}(e^{+};\delta)
        \cap
        B_{2}(e^{-};\delta)
        \cap
        \left[ 
            \dist({e^{+},e^{-}}) 
            \geq 
            t_{\omega}(e^{+})
            +
            t_{\omega}(e^{-})     
        \right]
        ,
    \end{equation}    
    and 
    \begin{equation}\label{eq:E:1}
        B_{3}
        =
        \bigcup_{e^+\in \cE^+}\bigcup_{e^-\in \cE^-}
        B_{3}(e^{+},e^-)
        .
    \end{equation}
	\item if \(\hnu_{+}>0\) and  \(\hnu_{-}=0\), let
    \begin{equation}\label{DKLRZ}
         B_{3}
        =
        \bigcup_{e^+\in \cE^+}
        B_{2}(e^{+})
    \end{equation}    
	\item if \(\hnu_{+}=0\) and  \(\hnu_{-}>0\), let
    \begin{equation}\label{CLPNW}
        B_{3}
        =
        \bigcup_{e^-\in \cE^-}
        B_{3}(e^{-})
    \end{equation}    
	\end{enumerate}
    Then \(\p{B_{3}}=1-o(1)\).
\end{lemma}
\begin{proof}
    We will only prove (i); case (ii) and (iii) are proved analogously.  Our proof is algorithmic and divided into two phases. Firstly we have the \emph{out-phase}, where we
    run the exploration process described in \autoref{sec:explore} repeatedly until the desired thin
    out-neighbourhood appears. Secondly we have the \emph{in-phase}, where we  run the exploration process in reversed direction repeatedly until we find a
    thin in-neighbourhood, disjoint from the previous one. Although the probability of success in each
    trial is small, we can show that the probability of eventual success goes to \(1\).
    Without loss of generality, we can assume that the half-edges are ordered in some
    arbitrary way.

    We provide more details. Let \(H_{0}^+\) be the empty partial pairing. At the $\ell$-th trial for
    $\ell\leq n^{1-\delta/4}$, choose \(e^{+}\) to be the smallest unpaired tail and  run the
    exploration process in \autoref{sec:explore} from \(e^{+}\), epoch by epoch, and conditioning on
    \(E_{H_{\ell-1}^+}\). Recall that the process terminates when we hit $H_{\ell-1}^+$ or when all
    tails that can be reached from $e^+$ have been paired. We add two extra termination conditions that
    are checked at the end of each epoch \(t\): (i) \(t \ge (1+\delta)t^{+}\) and (ii)
    \(\abs{\cN_{t}(e^{+})}\geq \omega\) (or equivalently, \(\abs{\cA_{i_t}^{+}}\geq \omega\) ). If
    the process terminates with condition (ii), \(t \ge (1-\delta) t^{+}\) and
    \(\abs{\cN_{t}(e^{+})}< \omega^2\), then $B_2(e^+)$ holds, we declare the $\ell$-th trial (and the out-phase) a success
    and proceed to the in-phase. Otherwise, we declare the trial a failure, obtain $H^+_\ell$ from
    $H^+_{\ell-1}$ by adding the new pairs, set $\ell$ to $\ell+1$ and restart the exploration
    process from the smallest unpaired tail. If
    \(\ell> n^{1-\delta/4}\), then we declare out-phase a failure and terminate. 

    If the out-phase succeeded, let \(H_{0}^-\) be the partial pairing obtained after the successful
    trial. We start the in-phase and run at most $n^{1-\delta/4}$ trials of the exploration process
   in reverse direction, where in the
    $\ell$-th trial we condition on $E_{H^-_{\ell-1}}$.

    If both phases succeed, then we have found $e^\pm$ with $B_{2}(e^{\pm})$, and by the
    definition of $H^-_\ell$, these neighbourhoods are disjoint.

    Let us compute the probability that the out-phase fails. Let \(F_{\ell}^{+}\) denote the event that the
    \(\ell\)-th trial in the out-phase failed, which implies the event \(B_{2}(e^{+})^c\).
    In each trial of the exploration process and regardless of whether it is a success or failure, at
    most \( O(\omega \log(n)) = O(\log^{7} n)\) half-edges are paired ({although the number of
    half-edges that have been activated might be larger}).  As there are at most  \(n^{1-\delta/4}\)
    trials, the event \(\cap_{j=1}^{\ell-1} F_{j}^{+}\) implies \(\abs{\cV(H_{\ell}^{+})} =
    O(n^{1-\delta/4} \log^{7}n) < n^{1-\delta/5}.\) It follows from \autoref{lem:tower:0} that
    regardless of \(\hnu_{+}>0\) or \(\hnu_{+}=0\)
    \begin{equation}\label{eq:low:fail}
        \begin{aligned}
            \p{F_{\ell}^{+} \cond \cap_{j=1}^{\ell-1} F_{j}^{+}}
            \le 
            1-n^{-1+\delta/2}
            .
        \end{aligned}
    \end{equation}
    (In the case \(\hnu_{+} = 0\), this is actually \(o(1)\).)
    Therefore
    \begin{equation}\label{eq:low:fail:all}
        \p{\cap_{\ell=1}^{\floor{n^{1-\delta/4}}} F_{i}^+}
        =
        \prod_{\ell=1}^{\floor{n^{1-\delta/4}}}
        \p{F_{\ell}^{+} \cond \cap_{j=1}^{\ell-1} F_{j}^{+}}
        \le
        \left( 1- n^{-1+  \delta/2} \right)^{n^{1-\delta/4}-1}
        \to 0
        .
    \end{equation}
Thus, the probability that the out-phase fails is \(o(1)\).  By a similar argument, the probability that
    the in-phase fails is also $o(1)$, concluding the proof.
\end{proof}

\subsection{Subcritical case}
\label{sec:sub}

When \(\nu \in (0,1)\), i.e., in the subcritical case, we have \(s_{\pm} = 0\) and \(\hnu_{\pm} =
\nu\). Thus \(t^{+}=t^{-} = \log_{1/\nu} n\).  By considering analogous events in branching process,
we can show that whp the neighbourhoods of all half-edges die before time \( (1+\delta)
t_{\pm}\) and that there exist half-edges \(e^{\pm}\) whose neighbourhood is of height \( (1-\delta) t^{\pm}\).

\begin{lemma}\label{EIYMO}
    Assume that \(0<\nu < 1\). Define the event \( A'(e^{\pm}, t) \coloneqq \left[ \abs{\cN_{t}(e^{\pm})} > 0
    \right].  \) 
    \begin{enumerate}[\normalfont(i)]
        \item Let \( B_{4} = B_{4}(\delta) = \cup_{e^{\pm} \in \cE^{\pm}} A'(e^{\pm},
            (1+\delta)t_{\pm}).  \) Then \(\p{B_{4}} = o(1)\).
        \item Let \( B_{5} = B_{5}(\delta) =\cup_{e^{\pm} \in \cE^{\pm}} A'(e^{\pm},
            (1-\delta)t_{\pm}).  \) Then \(\p{B_{5}} =1- o(1)\).
    \end{enumerate}
\end{lemma}

\begin{proof}
    The proof of (i) is analogous to that of \autoref{lem:tower:2}. Let \( (X_{t}^{\ua})_{t \ge 0}\)
    be a branching process with offspring distribution \(\Qnua\). Then it follows from
    \autoref{thm:BP:sub} and \autoref{lem:eq_T_GW} that
    \begin{equation}\label{COMCE}
        \p{A'(e^{\pm}, (1+\delta)t^{\pm})}
        \le
        (1+o(1))
        \p{X^{\ua}_{ (1+\delta) t^{\pm}} > 0} 
        \le 
        n^{-1-\delta/2}
        .
    \end{equation}
    Thus \(\p{B_{4}} = o(1)\) follows from a union bound over all half-edges.

    The proof of (ii) is analogous to that of \autoref{lem:tower:1}. We start with the smallest tail
    and explore its neighbourhood. If its neighbourhood either dies or reaches total size \(\omega\) before
    time \((1-\delta)t^{+}\), then we call it a failure and restart the exploration from the smallest
    unpaired tail.  Otherwise we call it a success and terminate.  The probability of success in one trial is
    at least \(n^{-1+\delta/2}\) by the same argument as in (i). Thus by repeating the process
    \(n^{1-\delta/4}\) times, whp we eventually succeed.
\end{proof}

\section{Distance between two sets of edges in the supercritical regime}\label{sec:typ:dist}

In this section, we show that whp, the distance between two modestly large sets of edges is about \(\log_{\nu}(n)\), given that $\nu>1$.

Let \(H\) be a partial pairing of half-edges as defined in \autoref{sec:explore}. Recall that \(\cP^{\pm}(H)\) is the set of paired heads and tails in \(H\),
\(\cV(H)\) is the set of nodes incident to half-edges in \(\cP^{\pm}(H)\), and \(E_{H}\) is the event that \(H\) is a subset of the pairing \(\vecGn\).

We consider a triplet \( (H, \cX^+,\cX^-) \) satisfying the following condition:
\begin{condition}\label{cond:H}
    $H$ is a partial pairing of $\cE^\pm$ and \(\cX^{\pm}\subseteq \cE^{\pm}(H)\setminus
    \cP^{\pm}(H)\).
\end{condition}

 For \(\cI \subseteq  [n]\), let \(\dist(\cX^+,\cX^-,\cI)\) be the minimal
    length of paths from \(\cX^+\) to \(\cX^-\) using only nodes in \(\cI\). 
The main result in this section is the following:
\begin{prop}\label{prop:connect}
   Uniformly over all choices of \(\varepsilon,\gamma >
    0\) and \( (H, \cX^+, \cX^-)\) satisfying \autoref{cond:H}, \(|\cV(H)| \le n^{1 -\gamma}\) and \(|\cX^+|,|\cX^-|\geq \omega\) we have
    \begin{equation}\label{eq:conn:1}
        \p{ \dist(\cX^+, \cX^-)   > (1+\varepsilon)\log_{\nu} n \cond E_H}
    =
    o(n^{-100})
    ,
    \end{equation}
    and, assuming in addition that \(|\cX^+|,|\cX^-|\le \omega^{2}\)
    \begin{equation}\label{eq:conn}
        \p{ \dist(\cX^+, \cX^-, [n]\setminus \cV(H))   < (1-\varepsilon)\log_{\nu} n \cond E_H}
    =
    o(n^{-\varepsilon/2})
    .
    \end{equation}
\end{prop}

\subsection{Lower bound by path counting}

We prove the lower bound \eqref{eq:conn} in \autoref{prop:connect} by a technique called \emph{path
counting} which was introduced by van der Hofstad in \cite{vanderhofstad2020}.

Given \(\cI \subseteq [n]\setminus \cV(H)\), 
a \emph{simple path} of length \(k\) from \(\cX^+\) to
\(\cX^-\) using nodes in \(\cI\) is a sequence
\begin{equation}
    \Pi =
    \left\{
        e^+,
        \left(v_1,e_1^-,e_1^+\right),
        \dots,
        \left(v_{k-1},e_{k-1}^-,e_{k-1}^+\right),
        e^-
    \right\}
    ,
\end{equation}
where \(e^+ \in \cX^+\), \(e^- \in \cX^-\), \(v_{i} \in\cI\) are distinct nodes, and \(e_{i}^-\in
\cE^-(v_i)\) and  \(e_{i}^+\in \cE^+(v_i)\).  Let \(P_{k}(\cX^+, \cX^-, \cI)\) be the number of
simple directed paths of length \(k\) from some \(\cX^+\) to some \(\cX^-\) only using nodes in
\(\cI\).    The following proposition is an adaptation of \cite[Proposition~7.4]{vanderhofstad2020}
for the directed configuration model  and provides an upper bound for the expected number of paths
of certain length from \(\cX^+\) to \(\cX^-\) conditioning on \(E_{H}\) using nodes outside
\(\cV(H)\):
\begin{lemma}
    \label{lem:Pk:mean}
    Let \(\cI \subseteq [n]\setminus\cV(H)\).
    Let \(r\) be the number of nodes \(i\in \cI\) with \(d_{i}^{+}d_{i}^{-}\geq 1\).   
    Let \(s=\abs{\cP^{-}(H)}\), i.e., the number of paired heads in \(H\).
    For any \(H\) any \(k \in [|\cI|{+1}]\), we have
    \begin{equation}\label{eq:Pk:up}
        n_{k,H}(\cX^+,\cX^-,\cI)
        \coloneqq
        \E{P_k(\cX^+,\cX^-,\cI) \cond E_H}
        \leq 
        \frac{\nu _{\mathcal{I}}^{k-1}  |\cX^+||\cX^-|}{m_n-k-s+1}
        \prod _{i=0}^{k-2} \frac{1-\frac{i}{r}}{1-\frac{i+s}{m_n}}
        ,
    \end{equation}
    where $\nu _{\mathcal{I}}$ is defined as in~\eqref{eq:nu:I}.
    (We use the convention that an empty product equals \(1\).)
\end{lemma}

\begin{proof}
    It suffices to consider only the case that \(\cX^+=\left\{ e^+ \right\}\), \(\cX^-=\left\{ e^-
        \right\}\) and show that
    \begin{equation}\label{eq:Pk:up:1}
        n_{k,H}(e^+,e^-,\cI)
        \le
        \frac{\nu _{\mathcal{I}}^{k-1}}{m_n-k-s+1}
        \prod _{i=0}^{k-2} \frac{1-\frac{i}{r}}{1-\frac{i+s}{m_n}}
        .
    \end{equation}
    Then \eqref{eq:Pk:up} follows by adding up the previous bound for all \(e^+\in \cX^+\) and \(e^-\in \cX^-\).

   Conditioning on \(E_H\), the exact probability for a given \(\Pi\) to exist is
    \begin{equation}
        \prod _{i=1}^k \frac{1}{m_n-i-s+1}
        .
    \end{equation}
    If we fix \(v_{1},\dots,v_{k-1}\), then the number of simple path using them in the given order is
    exactly
    \(
        \prod _{i=1}^{k-1} d^-(v_i) d^+(v_i)
        .
    \)
    Let \(\cI^k_*\) be the set of sequences of distinct nodes in \(\cI\) of length \(k\).
    Thus
    \begin{equation}\label{eq:Pk:up:2}
        n_{k,H}(e^+,e^-,\cI)
        =
        \frac{1}{m_n- k-s+1}
        \sum_{(v_{1},\dots,v_{k-1})\in\cI^{k-1}_*}
        \prod _{i=1}^{k-1} \frac{d^-(v_i) d^+(v_i)}{m_n-i-s+1}
        .
    \end{equation}
   
    Let \(\cR = \{ i\in \cI:\, d_{i}^{+}d_{i}^{-}\ge 1\}\). Note that \(\nu_\cR=\nu_\cI\). Define \(\cR^k_*\) as before. We use the following inequality of Maclaurin (\cite[Theorem 52]{hardy1988}), for \(r = |\cR|\),
    \(1 \le k-1 \le r\) and \((a_{i})_{i \in \cR}\) with \(a_{i} \ge 0\), we have
    \begin{equation}
        \frac{(r-k+1)!}{r!}
        \sum_{(\pi_{1},\dots,\pi_{k-1})\in\cR^{k-1}_*}
        \prod_{i=1}^{k-1}a_{\pi_i}
        \le
        \left( 
            \frac{1}{r}
            \sum_{i \in \cR} a_{i}
        \right)^{k-1}
        .
    \end{equation}
    Applying the above inequality with \(a_{\pi_i}=d^-(v_i) d^+(v_i)\), we have
    \begin{align*}
        n_{k,H}(e^+,e^-,\cI)
        &
        \leq
        \frac{1
        }{m_n-k-s+1  
        }
        \frac{1}{
            \prod _{i=1}^{k-1} m_n-i-s+1}
        \frac{r! }{(r-k+1)!}\left(\frac{\sum _{i\in \mathcal{R}} d^-(v_i) d^+(v_i)}{r}\right)^{k-1}
        \\
        &
        \leq \frac{\nu_{\mathcal{I}}^{k-1}}{m_n-k-s+1 }
        \prod _{i=0}^{k-2} \frac{1-\frac{i}{r}}{1-\frac{i+s}{m_n}}
        .
        \qedhere
    \end{align*}
\end{proof}

\begin{proof}[Proof of lower bound in \autoref{prop:connect}]
    Let \(k = \ceil{(1-\varepsilon)\log_{\nu} n}\) and $\cI=[n]\setminus \cV(H)$. Since
    $|\cV(H)|\leq n^{1-\gamma}$, by~\autoref{lem:s} we have $s=|\cP^-(H)|=o(n^{1-\gamma/2})$.
    Since $\abs{\cI}=n-o(n)$, it follows from \autoref{cor:s} that  \(\nu_{\cI} = (1+o(1)) \nu\).
    Let \(r\) be as in \autoref{lem:Pk:mean}. Then it follows from~\autoref{lem:Pk:mean} that
    \begin{equation}
        \begin{aligned}
            \p{\dist(\cX^+,\cX^-,\cI) \le k \cond E_H}
            &
            \le
            \sum_{\ell = 1}^{k} \E{P_{\ell}(\cX^+,\cX^-,\cI)\cond E_H}
            \\
            &
            \le
            \sum _{\ell=1}^{k}
            \frac{
                |\cX^+||\cX^-| 
                \nu _{\mathcal{I}}^{\ell-1} 
                \prod _{i=0}^{\ell-2} \frac{1-\frac{i}{r}}{1-\frac{i+s}{m_n}}
            }{
                m_n-\ell-s+1
            }
            \\
            &
            \le
            O(1)
            \frac{
                 |\cX^+||\cX^-| 
            }{
                m_n
            }
            \sum _{\ell=1}^{k}
            \nu _{\mathcal{I}}^{\ell-1} 
            \\
            &
            =
            O
            \left( 
                n^{-1} \omega^{4} \nu_{\cI}^{(1-\varepsilon)\log_{\nu} n}
            \right)
            =
            o(n^{-\varepsilon/2})
			.
        \end{aligned}
    \end{equation}
\end{proof}

\subsection{Upper bound by bounded expansion}

We prove the upper bound \eqref{eq:conn:1} in \autoref{prop:connect} by showing that a large set of
half-edges typically expands with rate at least \(\nu\), even after conditioning on $E_H$ for a small
partial pairing $H$. To this end, it suffices to consider
only nodes with bounded degree.

Given $\rho > 0$ that will be fixed later, choose \(K\) large enough so that
    \begin{equation}\label{eq:cond:K}
	\left( 
        1-\frac{\rho}{4}
    \right)
        \nu	 	
	\le    
    \E{\dinp \mathbb{1}\big({[\dinp<K] \cap [\dinm<K]}\big)}
    \le \nu
    .
\end{equation}
Note that such a \(K\) exists as \(\E{D_{\tin}^+}=\nu<\infty\) and \(D_{\tin}^+\) takes non-negative
values.  Let $\cL^\pm$ be the set of heads and tails incident to vertices with at least $K$ heads
or at least $K$ tails.

Given \( (H,\cX^+,\cX^-) \) satisfying \autoref{cond:H}, we want to explore the edge out-neighbourhoods of \(\cX^+\) and the edge in-neighbourhoods of \(\cX^-\) using only nodes with small in- and out-degree, and conditioning on
\(E_{H}\). Let $\cN^{*}_0(\cX^+)=\cX^+$ and for $k\geq 1$ define recursively
\begin{equation}\label{LLVTC}
\begin{aligned}
        \cN^{*}_k(\cX^+) = \{
            &
            e^+\in  \cE^+\setminus (\cL^+\cup \cN^{*}_{< k}(\cX^+)\cup \cF^+(H)) 
            \\
            &
            \mid  \exists f^+\in
        \cN^{*}_{k-1}(\cX^+), f^-\in \cE^-(v(e^{+})) , f^+f^- \text{ is a pair}\}
\end{aligned}
\end{equation}
In words, $\cN^{*}_k(\cX^+)$ is the set of tails that can be reached from $\cX^+$ by a path of
length $k$ using only vertices of low in- and out-degree. 
Note that the distance in $\vecGn$ from
$\cX^+$ to the elements of $\cN^{*}_k(\cX^+)$ is at most $k$. We define
$\cN^{*}_k(\cX^-)$ in a similar way.

\begin{lemma}\label{lem:expand}
    Uniformly over all choices of \(\gamma >
    0\), \( (H, \cX^+, \cX^-)\) satisfying \autoref{cond:H}, \(|\cV(H)| \le n^{1 -\gamma}\) and
    \(|\cX^+|,|\cX^-|\geq \omega\) we have:
    \begin{enumerate}[\normalfont(i)]
    \item
    for all 
    \(k \le \log _{(1+\rho)\nu }\left(n^{1-\gamma}/|\cX^+|\right)\)
    \begin{equation}\label{eq:A:k0}
        \p{  
            ((1-\rho)\nu)^{k}|\cX^+| \leq |\cN^{*}_k(\cX^+)| \leq  ((1{+}{\rho})\nu)^{k}|\cX^+|
            \cond E_H}=
        1-o(n^{-100})
        ;
    \end{equation}
    \item
     for all 
    \(k \le \log _{(1+\rho)\nu }\left(n^{1-\gamma}/|\cX^-|\right)\) 
    \begin{equation}\label{eq:B:k0}
        \p{  ((1-\rho)\nu)^{k} |\cX^-|\leq |\cN^{*}_k(\cX^-)|\leq  ((1{+}{\rho})\nu)^{k}
            |\cX^-|\cond E_H}=
        1-o(n^{-100})
        .
    \end{equation}  
    \end{enumerate}
\end{lemma}

\begin{proof}
We will show~\eqref{eq:A:k0}, then~\eqref{eq:B:k0} also holds by swapping $\cX^+$ and $\cX^-$ and reversing the direction of the edges. 

    For \(k\geq 1\),  let $d_k=|\cN^{*}_k(\cX^+)|$ and let \(E_k\) denote the event
    \begin{equation}\label{eq:E:k}
        \left(1-\rho\right)\nu
       d_{k-1}
        \leq d_{k}
        \leq 
        \left(1+\rho\right)\nu
        d_{k-1}
        .
    \end{equation}
    We show that uniformly for all \(k \le  \log_{(1+\rho)\nu}\left(n^{1-\gamma}/|\cX^+|\right)\),
    \begin{equation}\label{eq:A:k0:k}
        \p{E_k \left| E_{H} \cap \left[\cap_{j=1}^{k-1} E_j \right] \right.}
        = 1-o(n^{-1000})
        .
    \end{equation}
    from which~\eqref{eq:A:k0} follows directly.

 We may assume that $|\cX^+|\leq n^{1-\gamma}$, as otherwise there is nothing to prove. The event \(\cap_{i=1}^{k-1} E_i\) implies that $d_{k-1}\leq n^{1-\gamma}$ and $\sum_{i=0}^{k-1} d_i+|\cV(H)|= O(n^{1-\gamma})$.

Consider the exploration process defined in~\autoref{sec:explore} and introduce the following three modifications:
\begin{itemize}
    \item[-] initially, we let $\cA_0^+= \cX^+$ and $\cA_0^-=  \cE^-(\cV(\cX^{+}))$;
    \item[-] in  (iii), if $e^-_i\in \cL^-$, then we let \(\cA_{i}^{\pm} = \cA_{i-1}^{\pm}\setminus
        \{e_i^\pm\}\);
    \item[-] in (iii), if \(e_{i}^{-} \in \cF^{-}(H)\), we do not terminate the process and let  \(\cA_{i}^{\pm} = \cA_{i-1}^{\pm}\setminus
        \{e_i^\pm\}\).
\end{itemize}
Now the process generates a collection of rooted trees $\{\cT_{e^+}(i)\}_{e^+\in \cX^+}$. As in~\autoref{rem:distances}, the union of the \(k\)-th level of each tree is equal to $\cN_{k}^{*}(\cX^+)$ so it suffices to study the process.

Recall that $i_k$ is the last time we pair a tail at distance $k$ from $\cX^+$, so the $k$-th epoch of the
process goes from time $i_{k-1}+1$ to $i_{k-1}+d_{k-1}$, and $d_k$  is precisely the number of tails that
have been activated during this epoch. Let $X_i=|\cA_i|-|\cA_{i-1}|+1$. The only way we activate new
tails at the $i$-th step is if $e^-_i\in \cU^-_{i-1}\setminus \cL^-$, in which case we activate
$|\cE^+(v_i)|$ tails. Thus,
\begin{align*}
X_{i}=|\cE^+(v_i)|  \mathbb 1(e^-_i \in \cU^-_{i-1}\setminus \cL^-)
\end{align*} 
and
$$
d_k= \sum_{i=i_{k-1}+1}^{i_{k-1}+d_{k-1}} X_i.
$$

All the heads that were \emph{paired} or \emph{fatal} at the beginning of the process, are incident
to a vertex in
$\cV(H)$. All heads that have been paired before the \(k\)-th epoch are incident to a vertex
incident to  $\cN^{*}_{< k}(\cX^+)$ and there are at most $\sum_{i=0}^{k-1} d_i= O(n^{1-\gamma})$
such vertices. By applying~\autoref{lem:s} to the set of vertices incident to $\cE^-\setminus
\cU_{i-1}^-$ we have $|\cE^-\setminus \cU_{i-1}^-|= o(n^{1-\gamma/2})$.

Let \(\cH_{i-1}\) denote a history of the process before the \(i\)-th match that is
compatible with $E_H\cap \left[\cap_{j=1}^{k-1} E_j\right]$.
Then, for all \(\ell \in \{0,\dots, K-1\}\),
    \begin{align}
        \p{X_{i}=\ell \left| \cH_{i-1}\right.}
        &
        =
        \frac{\sum_{e^-\in \cU_{i-1}^-\setminus \cL^-} \mathbb 1(d^{+}(v(e^-))=\ell)}{m_n-|\cP^-(H)|-(i-1)}
        \\
        &
 		\geq
        \frac{\sum_{e^-\in \cE^-\setminus \cL^-} \mathbb 1(d^{+}(v(e^-))=\ell)}{m_n}-\frac{|\cE^-\setminus \cU_{i-1}^-|}{m_n}
        \\       
        &
        \geq
        \p{\left[ \dninp=\ell \right] \cap \left[ \dninm < K \right]}
        - 
        n^{-\gamma/4}
        \eqqcolon
        b_{n,\ell}
        .
    \end{align}

    Let \(\bar{X}_{i}\) and \(\hat{X}_{i}\) be two independent random variables with distributions
    \begin{equation}
        \p{\bar{X}_{i}=\ell}
        =
        \begin{cases}
            1- \sum_{j=1}^{K} (b_{n,j} \vee 0), 
            &
            \text{if }\ell=0
            \\
            b_{n,\ell} \vee 0,
            &
            \text{if }1\leq \ell<K
            \\
            0
            &
            \text{if }\ell\geq K
        \end{cases}
        ,
    \end{equation}
    and
    \begin{equation}
        \p{\hat{X}_{i}=\ell}
        =
        \begin{cases}
            b_{n,\ell} \vee 0,
            &
            \text{if }0\leq \ell<K-1
            \\
            1- \sum_{j=0}^{K-1} (b_{n,j} \wedge 0), 
            &
            \text{if }\ell=K-1
            \\
            0
            &
            \text{if }\ell\geq K
        \end{cases}
        .
    \end{equation}
    Then \(\bar{X_{i}} \le_{\text{st}} \left(X_{i} \left| \cH_{i-1}\right.\right)\le_{\text{st}} \hat{X}_{i}\). Moreover,
    \begin{align*}\label{eq:bar:X}
        \E{\bar{X}_{i}}
        =
        (1+o(1))
        \E{\dinp \mathbb{1}_{[\dinp<K] \cap [\dinm<K]}}
        -
        Kn^{-\gamma/4}
        \ge \left( 1- \frac{\rho}{2} \right) \nu
        ,
    \end{align*}
    where the last step follows from our choice of \(K\) that satisfies~\eqref{eq:cond:K}.   Similar computations show that
    \(
        \E{\hat{X}_{i}}
        \le
          \left(1+{\rho}/{2}\right)
          \nu
        .
    \)

    Note that both \(\hat{X}_i\) and \(\bar{X}_i\) are bounded random variables so we 
    can applying Hoeffding's inequality \cite{hoeffding1963a}. It follows that
    \begin{align*}
       \p{d_k < (1-\rho){\nu}d_{k-1} \cond \cH_{i_{k-1}}}
        &
        \le
        \p{\sum_{i=i_{k-1}+1}^{i_{k-1}+d_{k-1}} (\bar{X}_{i}- \e \bar{X}_i) >
          (\rho/2) \nu d_{k-1}
        \cond \cH_{i_{k-1}}}
        \\
        &
        \leq \exp \left(-\frac{\rho^2\nu d_{k-1} }{8 K^2}\right)
        =
        o\left( n^{-1000} \right).
    \end{align*}
    where we used that $d_{k-1}\geq d_0\geq \omega$ by the conditioning \(\cH_{k-1}\). Similarly, using $\hat X_i$, we obtain
    \begin{align*}
        &
        \p{d_k > (1+\rho){\nu}d_{k-1} \cond \cH_{i_{k-1}} }
        =
        o\left( n^{-1000} \right).
    \end{align*}
    This proves \eqref{eq:A:k0:k} and \eqref{eq:A:k0}. The proof of \eqref{eq:B:k0} is 
    analogous.
\end{proof}

\begin{proof}[Proof of upper bound in \autoref{prop:connect}]
    Let $\beta,\rho>0$. Let \(k^{+}:=\lceil \log_{(1-\rho)\nu} (n^{1/2+\beta}/|\cX^+|)\rceil\)  and
    \(k^{-}:= \lceil\log_{(1-\rho)\nu} (n^{1/2+\beta}/|\cX^-|)\rceil\) .  By choosing $\beta,\rho$
    small enough with respect to $\varepsilon$, we have $k^{+}+k^{-}+1\leq (1+\varepsilon)\log_\nu
    n$.  
    It follows from \autoref{lem:expand} that with probability \(1-o(n^{-100})\),
    \begin{equation}
        |\cN^{*}_{k^{+}}(\cX^+)| \ge \left((1-\rho) \nu\right)^{k^+}|\cX^+| \geq n^{\frac{1}{2}+\beta},
    \end{equation}
	and similarly $ |\cN^{*}_{k^{-}}(\cX^-)| \ge n^{\frac{1}{2}+\beta}$.
    
    If a tail in \(\cN^{*}_{< k^{+}}(\cX^+)\) has been paired with a head in \(\cN^{*}_{<
    k^{-}}(\cX^-)\), then \(\dist(\cX^+, \cX^-) \le  k^+ +k^-\).  Otherwise, the probability
    that no tail in \(\cN^{*}_{k^{+}}(\cX^+)\) is paired to a head in \(\cN^{*}_{k^{-}}(\cX^-)\) is
    at most
    \begin{equation}
        \left( 
            1- \frac{
                n^{1/2+\beta}
            }{m_n}
        \right)
        ^{
            n^{1/2+\beta},
        }
        =
        o(n^{-100})
        .
    \end{equation}
    Therefore, that probability of \(\dist(\cX^+,\cX^-) > k^{+}+k^{-}+1\) is \(o(n^{-100})\). 
\end{proof}

\subsection{Distance between thin neighbourhoods}

{We will use~\autoref{prop:connect} to show that the distance between thin neighbourhoods are about
\(\log_{\nu} n\).}
\begin{lemma}\label{lem:no:towers}
    Assume that \(\hnu_{\pm} > 0\).
    Let \(t^{+}\), \(t^{-}\) be as in \eqref{EOSJR} and let \(B_{3}\) be as in
    \autoref{lem:tower:1}.
    Let
    \begin{equation}\label{XLQSQ}
        k_n= t^{+}+ t^{-}+\log_\nu n.
    \end{equation}
    Let
    \begin{equation}
        B_{4}(e^+,e^-)
        =
        B_{4}(e^+,e^-;\varepsilon)
        = B_{3}(e^+,e^-;\varepsilon/6) \cap 
        \left[ 
            \abs{\frac{\dist(e^+,e^-)}{k_n}-1} > \varepsilon
        \right]
        ,
    \end{equation}
    and let
    \begin{equation}\label{NMVRB}
        B_{4}=\bigcup_{e^+\in \cE^+} \bigcup_{e^-\in \cE^-} B_{4}(e^+,e^-)
    \end{equation}   
    Then \(
        \p{
            B_{4}
        }
        =
        o(1)        
        .
    \)
\end{lemma}
\begin{proof}
Fix $e^+\in \cE^+$ and $e^-\in \cE^-$. It suffices to show that $\p{B_{4}(e^+,e^-)}=o(n^{-2})$.

We have 
\begin{equation}\label{eq:B4:1}
    \p{B_{4}(e^+,e^-)}
    = 
    \p{B_{3}\left(e^+,e^-;\frac{\varepsilon}{6}\right)}
    \p{
        \abs{\frac{\dist(e^+,e^-)}{k_n}-1} > \varepsilon 
        \cond 
        B_{3}\left(e^+,e^-;\frac{\varepsilon}{6}\right)
    }
    .
\end{equation}
By applying (i) of \autoref{lem:tower:0} twice where in the second time we condition on $E_{H_0}$, where $H_0$ is the partial pairing resulting from the exploration of the out-neighbourhoods of $e^+$ and satisfies $|\cV(H_0)|=O(\log^7 n)$, we obtain 
\begin{equation}\label{eq:B4:2}
\p{B_{3}(e^+,e^-;\varepsilon/6)} \leq n^{-2+\epsilon/2}
\end{equation} 
Note that by the choice of $H$, the two neighbourhoods are disjoint.

Let $H$ be the partial pairing of the edges exposed during the previous exploration process,
conditional on $B_{3}(e^+,e^-)$, so $|\cV(H)|=O(\log^7 n)$.  Let
$\cX^+=\cN_{t_\omega(e^+)}(e^+)$ and $\cX^-=\cN_{t_\omega(e^-)}(e^-)$. Note that
$(H,\cX^+,\cX^-)$ satisfies~\autoref{cond:H} and $|\cX^+|,|\cX^-|\in [\omega,\omega^2)$.  Since
$t_\omega(e^\pm)\leq (1+\varepsilon/6)t^\pm$, by applying \eqref{eq:conn:1}
in \autoref{prop:connect}, 
\begin{align}
\p{\dist(e^+,e^-)> (1+\varepsilon)k_n\cond E_H} 
&
=\p{\dist(\cX^+,\cX^-)> (1+\varepsilon) \log_\nu n\cond E_H}
\\ 
&
=
o(n^{-100})
.
\end{align}

Note that there is no simple path from $\cX^+$ to $\cX^-$ that uses vertices in $H$.  Since
$t_\omega(e^\pm)\geq (1-\varepsilon/6)t^\pm$, if follows from \eqref{eq:conn} in
\autoref{prop:connect} that
\begin{align}
\p{\dist(e^+,e^-)< (1-\varepsilon) k_n\cond E_H} 
&\leq 
\p
{
    \dist(\cX^+,\cX^-,[n]\setminus\cV(H))< (1-\varepsilon) \log_\nu n\cond E_H
}
\\
&=
o(n^{-\varepsilon/2})
.
\end{align}
As this is true for any $H$, we have
\begin{equation}\label{eq:B4:3}
\p{\abs{\frac{\dist(e^+,e^-)}{k_n}-1} > \varepsilon \cond B_{3}(e^+,e^-)}= o(n^{-\epsilon/2})
\end{equation}
and the lemma follows by putting~\eqref{eq:B4:2} and~\eqref{eq:B4:3} in~\eqref{eq:B4:1}.
\end{proof}

Next lemma holds in the case where in- or out- thin neighbourhoods do not exist.
\begin{lemma}\label{lem:no:towers2}
    Let \(t^{+}\), \(t^{-}\) be as in \eqref{EOSJR}, $k_n= t^{+}+t^-+ \log_\nu n$ and let \(B_{2}\) be as in
    \autoref{lem:tower:0}. Then
    \begin{enumerate}[\normalfont(i)]
    \item if $\hnu_+>0$ and $\hnu_-=0$ (so $t^-=0$), for every $\cX^-\subseteq \cE^-$ with $|\cX^-|\in [\omega,\omega^2]$ define \begin{equation}
        B_{4}(e^+,\cX^-)
        =
        B_{4}(e^+,\cX^-;\varepsilon)
        = B_{2}(e^+;\varepsilon/3) \cap 
        \left[ 
            \abs{\frac{\dist(e^+,\cX^-)}{k_n}-1} > \varepsilon
        \right],
    \end{equation}
    and $B_{4}(\cX^-)=\bigcup_{e^+\in \cE^+}  B_{4}(e^+,\cX^-)$. Then \(\p{
            B_{4} (\cX^-)}=o(1)\).
    \item if $\hnu_+=0$ and $\hnu_->0$ (so $t^+=0$), for every $\cX^+\subseteq \cE^+$ with $|\cX^+|\in [\omega,\omega^2]$ define \begin{equation}
        B_{4}(e^-,\cX^+)
        =
        B_{4}(e^-,\cX^+;\varepsilon)
        = B_{2}(e^-;\varepsilon/3) \cap 
        \left[ 
            \abs{\frac{\dist(\cX^+,e^-)}{k_n}-1} > \varepsilon
        \right],
    \end{equation}
    and $B_{4}(\cX^+)=\bigcup_{e^-\in \cE^-}  B_{4}(e^-,\cX^+)$. Then \(\p{
            B_{4} (\cX^+)}=o(1)\).
    \end{enumerate} 
\end{lemma}
\begin{proof}[Sketch of the proof]
We only sketch the proof of (i) as both proofs are analogous and similar to the proof
of~\autoref{lem:no:towers}. We apply~\autoref{lem:tower:0} (i) only once to upper bound the
probability of $B_{2}(e^-;\varepsilon/3)$ by $n^{-1+\epsilon/2}$, so such tails are rare but
possible. Then we let $\cX^+=\cN_{t_\omega(e^+)}(e^+)$ and we use~\autoref{prop:connect} to connect
$\cX^+$ and $\cX^-$ whp with $H$ being the partial pairing resulting from the exploration of $\cN_{\leq
t_\omega(e^+)}(e^+)$. 
\end{proof}

\section{Diameter}\label{sec:final}

With all the preparation at hand, the proof of \autoref{thm:diam} is readily available.
%We only prove (i) of \autoref{thm:diam} in which case \(\nu > 1\) since (ii) has already been proved
%at the end of \autoref{sec:sub}.

\subsection{Supercritical: Lower bound}\label{sec:final:up}

We split into cases depending on $\hnu_{+}$ and $\hnu_{-}$. If $\hnu_{+},\hnu_{-}>0$,
by~\autoref{lem:tower:1} (i), whp there exist a tail \(e^{+}\) and a head \(e^{-}\) satisfying
$B_3(e^+,e^-;\varepsilon/6)$. By~\autoref{lem:no:towers} whp there is no such pair also satisfying
$\dist(e^+,e^-)\notin ((1- \varepsilon)k_n,\infty)$, where $k_n= t^{+}+t^-+ \log_\nu n$.

If $\hnu_{+}>0$ and $\hnu_{-}=0$, then fix an arbitrary set of heads $\cX^-$ with $|\cX^-|\in
[\omega,\omega^2]$. By \autoref{lem:tower:1} (ii), there exists  a tail $e^+$ satisfying $B_2(e^+)$,
but by~\autoref{lem:no:towers2} (i), no such tail satisfies $\dist(e^+,\cX^-)\notin ((1-
\varepsilon)k_n,\infty)$. The proof is analogous if $\hnu_{+}=0$ and $\hnu_{-}>0$.

If $\hnu_{+}=\hnu_{-}=0$ (so $t^+=t^-=0$), then we fix two arbitrary sets $\cX^+\subseteq \cE^+$ and
$\cX^-\subseteq \cE^-$  with $|\cX^+|,|\cX^-|\in [\omega,\omega^2]$. By~\autoref{prop:connect} with
$H$ the empty pairing, we obtain $\dist(\cX^+,\cX^-)\in (1\pm \varepsilon)\log_\nu n= (1\pm
\varepsilon)k_n$.

In each case we obtain the existence of $e^+\in \cE^+$ and $e^-\in \cE^-$ at distance in
$((1-\varepsilon)\log n,\infty)$. 
Let \(v^{+}\) be the node incident to the head paired with \(e^{+}\).
Let \(v^{-}\) be the node incident to the tail paired with \(e^{-}\). Then \( \dist(v^{+}, v^{-}) =
\dist(e^{+}, e^{-}) -2, \) concluding the proof of the lower bound in~(i) of \autoref{thm:diam}.

\subsection{Supercritical: Upper bound}

Assume first that \(\hnu_{+}>0\) and \(\hnu_{-}>0\).  By \autoref{lem:tower:2}, for every pair of
half-edges \(e^{+} \in \cE^{+}\) and \(e^{-} \in \cE^{-}\), whp $B_1^c$ holds; that is, either
{$C_1$}: there are no edges at distance more than \( (1+\varepsilon) t^+\) from \(e^{+}\), or
{$C_2$}: there are no edges at distance more than \( (1+\varepsilon) t^-\) to \(e^{-}\), or
{$C_3$}:
\begin{align*}
    &
    t_{\omega}(e^{+}) <
    (1+\varepsilon) t^+
    ,
    \qquad
    t_{\omega}(e^{-}) <
    (1+\varepsilon) t^-
    ,
\end{align*}
and,
\begin{equation}
    \abs{
       \cX^+
    }
    > \omega
    ,
    \qquad
    \abs{
        \cX^-
    }
    > \omega
    ,
\end{equation}
 where $\cX^+= \cN_{t_{\omega}(e^{+})}(e^{+})$ and $\cX^-=\cN_{t_{\omega}(e^{-})}(e^{-})$.

{If $C_1\cup C_2$ holds}, then either
\begin{equation}
    \dist(e^{+},e^{-}) < 
    (1+\varepsilon) (t^+\wedge t^-)
    ,
\end{equation}
or there is no path from  \(e^{+}\) to \(e^{-}\)  and \(\dist(e^{+},e^{-})=\infty\).

Suppose that $C_3$ holds.  If a tail in \(\cN_{<t_\omega(e^+)}(e^+)\) has been paired with a head in \(\cN_{< t_\omega(e^-)}(e^-)\), then \(\dist(e^+, e^-) < (1+\varepsilon)(t^++t^-)\) and we are done.  
Otherwise, let $H$ be the partial pairing induced by $\cN_{<t_\omega(e^+)}(e^+)$ and \(\cN_{< t_\omega(e^-)}(e^-)\). Since $(H,\cX^+,\cX^-)$ satisfies~\autoref{cond:H} and $|\cX^+|,|\cX^-|>\omega$, it follows from \autoref{prop:connect} that with probability \(o(n^{-100})\),
\(\dist(\cX,^+\cX^-) > (1+\varepsilon)\log_\nu n\). Thus
\begin{align*}    
    \p{ 
    (1+\varepsilon) 
    \left( 
        t^++t^-+\log_\nu n
    \right)
    < \dist(e^{+},e^{-}) <\infty
        \cond C_3}
    &
    \le
    \p{
        \dist(\cX^-,\cX^-) >(1+\varepsilon)\log_\nu n
        \cond C_3
    }\\
    &
    = o(n^{-100})
    .
\end{align*}
The upper bound for the diameter follows from applying a union bound over all \(e^{\pm} \in
\cE^{\pm}\).

In the case \(\hnu_{\pm} = 0\), the above argument still works by replacing \(
(1+\varepsilon)t^{\pm}\) by \(\varepsilon \log n\).

\subsection{Subcritical}

Part (ii) of \autoref{thm:diam} follows immediately from \autoref{EIYMO}.  By (i) of \autoref{EIYMO}, whp for any pair of
nodes \( (u,v)\) in \(\vecGn\), either \(\dist(u,v) < (1+\delta) \log_{1/\nu} n\) or \(\dist(u,v) =
\infty\).  By (ii) of \autoref{EIYMO}, whp there exist a tail \(e^{+}\) and a head \(e^{-}\) such
that \(\dist(e^{+},e^{-}) \in ( (1-\delta)\log_{{1}/{\nu}} n, \infty)\).   Let \(v^{+}\) be the
node incident to the head paired with \(e^{+}\).  Let \(v^{-}\) be the node incident to the tail
paired with \(e^{-}\). Then \( \dist(v^{+}, v^{-}) = \dist(e^{+}, e^{-}) -2\).

\section{Applications}
\label{sec:app}

In this section, we give some applications of our results in the supercritical regime without delving into too much details.

\subsection{Typical distance}
\label{sec:app:typical}

Let \(U_{1}, U_{2} \in [n]\) be two vertices chosen uniformly at random.  Then
\(\dist(U_{1}, U_{2})\) is called the \emph{typical distance} of \(\vecGn\).  A distributional
result of \(\dist(U_{1},U_{2})\) is given by  van der Hoorn and Olvera-Cravioto~\cite{hoorn2018}.  
Here we give a weaker result under weaker assumptions:
\begin{thm}\label{thm:typical}
    Assume \autoref{cond:main} and \(\nu > 1\).
    Let \(U_{1},U_{2} \in [n]\) be two vertices chosen uniformly at random in \(\vecGn\). Then for all
    \(\varepsilon>0\),
    \begin{equation}\label{eq:typical:low}
        \p{
            \abs{
                \frac{
                    \dist{}(U_{1},U_{2})
                }{
                    \log_{\nu} n
                }
                -1
            }
            <
            \varepsilon
            \cond
            \dist(U_1,U_2) < \infty
        }
        \to
        1
        .
    \end{equation}
\end{thm}

The proof of the theorem is an easy application of \autoref{prop:connect} and we leave it to the reader.

\subsection{Other random graphs}

In many random digraphs, the degree sequence is not fixed but random. However, many such models,
including \(d\)-out regular digraphs and binomial random digraph described below, can be studied via
the directed configuration model using the following simple lemma whose proof we omit:

\begin{lemma}\label{NYDQE}
    Assume \(\dsD_{n}\) is a random directed multi/simple graph of \(n\) vertices which is
    uniformly random conditioning on its degree sequence. Let \(D_{n}\) be the in- and out-degree of
    a uniform random vertex in \(\dsD_{n}\) and assume that \(D_{n}\) satisfies \autoref{cond:main}
    with some distribution \(D\) on \(\dsZ_{\ge 0}^{2}\). Let
    \(\vecGn\) be the directed configuration model satisfying \autoref{cond:main} with the same
    \(D\). If \(\vecGn\) has a property \(P_{n}\) whp, then \(\dsD_{n}\) has property
    \(P_{n}\) whp.
\end{lemma}

\subsubsection{Regular digraphs}

The $d$-out model \(\dsD_{n,d\text{-}\tout}\) is a directed multigraph on \([n]\) in which each of
vertex is given $d$ out-edges whose end verteices are chosen independently and uniformly at random
from all vertices.  In this model, \(D_{n}\) converges in distribution, first moment, and second
moment to \((D^{+},D^{-})\), where \(D^{+} \equiv d\) and \(D^{-} \eqd \Poi(d)\).  When \(d \ge 2\), whp there
are thin in-neighbourhoods but no thin out-neighbourhood in \(\dsD_{n,d\text{-}\tout}\).  Therefore,
we recover the following result in \cite{addario2015diameter} by applying \autoref{thm:diam}, and
\autoref{NYDQE}:

\begin{thm}\label{thm:dout}
    Assume that \(d \ge 2\).  Let \(\lambda_d\) be the unique solution of \(d e^{-d} = \lambda_{d}
            e^{-\lambda_{d}}\) on \( (0,1)\). Then
    \begin{equation}\label{NLWTF}
        \frac{\diam{}(\dsD_{n,d\text{-}\tout})}{\log n}
        \to
        \frac{1}{\log(1/\lambda_{d})}
        +
        \frac{1}{\log d}
        ,
    \end{equation}
    in probability.
\end{thm}

A related model is the \(d\)-in/out model \(\dsD_{n,d}\), i.e., the uniform random directed multigraph on \([n]\)
in which each vertex has both in- and out-degree \(d\). When \(d \ge 2\), we have neither
thin out-neighbourhood nor thin in-neighbourhood and the diameter is of the same order as the
typical distance.

\begin{thm}\label{thm:dout:1}
    Assume that \(d \ge 2\).  Then
    \begin{equation}\label{MGBEP}
        \frac{\diam{}(\dsD_{n,d})}{\log n}
        \to
        \frac{1}{\log d}
        ,
    \end{equation}
    in probability.
\end{thm}
\subsubsection{Binomial random digraph}

A binomial random digraph \(\dsD_{n,p}\) is a simple digraph on \([n]\) in which a directed edge is added
between each ordered pair of vertices independently with probability \(p\), as described in
\cite{karp1990a, luczak2009}.  

A slightly different model  \(\dsD_{n,p}^{*}\) recently introduced by Ralaivaosaona, Rasendrahasina
and Wagner~\cite{ralaivaosaona2020} is constructed by adding an \emph{undirected} edge between each pair
of vertices independently with probability \(2p\) and choosing the direction of the edge with a fair
coin toss.

In both models, assuming that \(n p \to \nu > 1\), \(D_{n}\) converges in distribution
to a pair of independent Poisson random variables with expectation \(\nu\).  Thus, we
have the following result by applying \autoref{thm:diam}, 
and \autoref{NYDQE}:

\begin{thm}\label{thm:binom}
    Assume that \(n p \to \nu > 1\).
    Let \(\hnu\) be the unique solution of \(\hnu e^{-\hnu} = \nu
    e^{-\nu}\) on \( (0,1)\). Then
    \begin{equation}\label{AOPLN}
        \frac{\diam{}(\dsD_{n,p})}{\log n}
        \to
        \frac{2}{\log(1/\hnu)}
        +
        \frac{1}{\log \nu}
        ,
    \end{equation}
    in probability,
    and
    \begin{equation}\label{GLFAC}
        \frac{\diam{}(\dsD_{n,p}^{*})}{\log n}
        \to
        \frac{2}{\log(1/\hnu)}
        +
        \frac{1}{\log \nu}
        ,
    \end{equation}
    in probability.
\end{thm}

\section*{Acknowledgements}

We thank Svante Janson for helpful discussions on branching processes in the subcritical regime.

\bibliographystyle{abbrvnat}
\bibliography{rdg}

\end{document}